\documentclass[reqno]{amsart}

\usepackage{tikz-cd}

\usepackage{fullpage,dsfont}

\usepackage{hyperref} 

\usepackage{parskip}
\makeatletter 
\def\thm@space@setup{%
 \thm@preskip=\parskip \thm@postskip=0pt
}
\def\th@remark{%
  \thm@headfont{\itshape}%
  \normalfont 
  \thm@preskip\parskip \thm@postskip=0pt
}
\makeatother

\usepackage[nobysame,alphabetic,initials,msc-links]{amsrefs}

\usepackage[final]{pdfpages}

\usepackage{multirow}

\usepackage{array}
\usepackage{kbordermatrix}

\DefineSimpleKey{bib}{how}
\DefineSimpleKey{bib}{mrclass}
\DefineSimpleKey{bib}{mrnumber}
\DefineSimpleKey{bib}{fjournal}
\DefineSimpleKey{bib}{mrreviewer}

\renewcommand{\PrintDOI}[1]{%
  \href{http://dx.doi.org/#1}{{\tt DOI:#1}}%
}
\renewcommand{\eprint}[1]{#1}
\BibSpec{book}{%
    +{}  {\PrintPrimary}                {transition}
    +{.} { \PrintDate}                  {date}
    +{.} { \textit}                     {title}
    +{.} { }                            {part}
    +{:} { \textit}                     {subtitle}
    +{,} { \PrintEdition}               {edition}
    +{}  { \PrintEditorsB}              {editor}
    +{,} { \PrintTranslatorsC}          {translator}
    +{,} { \PrintContributions}         {contribution}
    +{,} { }                            {series}
    +{,} { \voltext}                    {volume}
    +{,} { }                            {publisher}
    +{,} { }                            {organization}
    +{,} { }                            {address}
    +{,} { }                            {status}
    +{,} { \PrintDOI}                   {doi}
    +{,} { \PrintISBNs}                 {isbn}
    +{}  { \parenthesize}               {language}
    +{}  { \PrintTranslation}           {translation}
    +{;} { \PrintReprint}               {reprint}
    +{.} { }                            {note}
    +{.} {}                             {transition}
    +{}  {\SentenceSpace \PrintReviews} {review}
}
\BibSpec{article}{%
    +{}  {\PrintAuthors}                {author}
    +{,} { \textit}                     {title}
    +{.} { }                            {part}
    +{:} { \textit}                     {subtitle}
    +{,} { \PrintContributions}         {contribution}
    +{.} { \PrintPartials}              {partial}
    +{,} { }                            {journal}
    +{}  { \textbf}                     {volume}
    +{}  { \PrintDatePV}                {date}
    +{,} { \issuetext}                  {number}
    +{,} { \eprintpages}                {pages}
    +{,} { }                            {status}
    +{,} { \PrintDOI}                   {doi}
    +{,} { \eprint}        {eprint}
    +{}  { \parenthesize}               {language}
    +{}  { \PrintTranslation}           {translation}
    +{;} { \PrintReprint}               {reprint}
    +{.} { }                            {note}
    +{.} {}                             {transition}
    +{}  {\SentenceSpace \PrintReviews} {review}
}
\BibSpec{collection.article}{%
    +{}  {\PrintAuthors}                {author}
    +{,} { \textit}                     {title}
    +{.} { }                            {part}
    +{:} { \textit}                     {subtitle}
    +{,} { \PrintContributions}         {contribution}
    +{,} { \PrintConference}            {conference}
    +{}  {\PrintBook}                   {book}
    +{,} { }                            {booktitle}
    +{,} { \PrintDateB}                 {date}
    +{,} { pp.~}                        {pages}
    +{,} { }                            {publisher}
    +{,} { }                            {organization}
    +{,} { }                            {address}
    +{,} { }                            {status}
    +{,} { \PrintDOI}                   {doi}
    +{,} { \eprint}        {eprint}
    +{}  { \parenthesize}               {language}
    +{}  { \PrintTranslation}           {translation}
    +{;} { \PrintReprint}               {reprint}
    +{.} { }                            {note}
    +{.} {}                             {transition}
    +{}  {\SentenceSpace \PrintReviews} {review}
}
\BibSpec{misc}{%
  +{}{\PrintAuthors}  {author}
  +{,}{ \textit}      {title}
  +{.}{ }             {how}
  +{}{ \parenthesize} {date}
  +{,} { available at \eprint}        {eprint}
  +{,}{ available at \url}{url}
  +{,}{ }             {note}
  +{.}{}              {transition}
}
\usepackage{amssymb, amsfonts, amsxtra, amsmath}
\usepackage{mathrsfs}
\usepackage{mathdots}
\usepackage{wasysym}
\usepackage[all]{xy}
\usepackage{bbm}
\usepackage{calc}
\usepackage{accents}

\usepackage{stackengine}

\numberwithin{equation}{section}

\newtheorem{TheoIntro}{Theorem}

\newtheorem{Theorem}{Theorem}[section]
\newtheorem*{Theorem*}{Theorem}
\newtheorem{Def}[Theorem]{Definition}
\newtheorem*{Def*}{Definition}
\newtheorem{Lem}[Theorem]{Lemma}
\newtheorem{Prop}[Theorem]{Proposition}
\newtheorem{Cor}[Theorem]{Corollary}

\newtheorem{Rem}[Theorem]{Remark}

\newtheorem{Exa}[Theorem]{Example}

\newcommand\bp{\begin{proof}}
\newcommand\ep{\end{proof}}

\mathchardef\mhyph="2D

\DeclareMathOperator{\Ad}{\mathrm{Ad}}

\DeclareMathOperator{\id}{\mathrm{id}}

\DeclareMathOperator{\Ker}{\mathrm{Ker}}

\DeclareMathOperator{\Det}{\mathrm{Det}}
\DeclareMathOperator{\Sym}{\mathrm{Sym}}

\newcommand{\cop}{\mathrm{cop}}

\newcommand{\wt}{\mathrm{wt}}

\newcommand{\msZ}{\mathscr{Z}}

\newcommand{\mfu}{\mathfrak{u}}

\newcommand{\mcH}{\mathcal{H}}

\newcommand{\mcO}{\mathcal{O}}

\newcommand{\mbr}{\mathbf{r}}

\newcommand{\mbs}{\mathbf{s}}

\newcommand{\brr}{[r]}
\newcommand{\brk}{[k]}

\newcommand{\brN}{[N]}

\newcommand{\C}{\mathbb{C}}

\newcommand{\N}{\mathbb{N}}

\newcommand{\R}{\mathbb{R}}
\newcommand{\T}{\mathbb{T}}

\newcommand{\Z}{\mathbb{Z}}

\newcommand{\res}{\mathrm{res}}

\newcommand{\br}{\mathrm{br}}

\title{Representation theory of the reflection equation algebra III: Classification of irreducible representations}

\author{Stephen T. Moore}
\address{Institute of Mathematics, Polish Academy of Sciences}
\address{Now at: Xi'an Jiaotong-Liverpool University}
\email{stm862@gmail.com}

\thanks{The work of S.T.M.~was supported by the FWO grants G025115N and G032919N. S.T.M. was additionally supported by Narodowe Centrum Nauki, grant number 2017/26/A/ST1/00189.}

\begin{document}
\maketitle

\begin{abstract}
We continue our study of Hilbert space representations of the Reflection Equation Algebra, again focusing on the algebra constructed from the $R$-matrix associated to the $q$-deformation of $GL(N,\C)$ for $0<q<1$. We develop a form of highest weight theory and use it to classify the irreducible bounded $*$-representations of the reflection equation algebra.
\end{abstract}

\section*{Introduction and statement of the main result}

Let  $0<q<1$, and let $\hat{R} \in M_N(\C) \otimes M_N(\C)$ be the \emph{braid operator}
\begin{equation}\label{EqBraidOp}
	\hat{R} =  \sum_{ij} q^{-\delta_{ij}}e_{ji}\otimes e_{ij} + (q^{-1} -q)\sum_{i<j} e_{jj} \otimes e_{ii},
\end{equation}
satisfying the \emph{braid relation}
\[
\hat{R}_{12}\hat{R}_{23}\hat{R}_{12} = \hat{R}_{23} \hat{R}_{12}\hat{R}_{23},
\] 
where $\hat{R}_{23}$ denotes $1\otimes \hat{R}$, etc. $\hat{R}$ also satisfies the \emph{Hecke relation}
\[
(\hat{R}-q^{-1})(\hat{R}+q)  = \hat{R}^2 +(q-q^{-1})\hat{R} -1 = 0
\]
and is self-adjoint, $\hat{R}^* = \hat{R}$.

\begin{Def*}
	The \emph{Reflection Equation Algebra} (\emph{REA}) is the universal unital complex algebra $\mcO_{q}^{\br}(M_N(\C))$ generated by the matrix entries of 
	\[
	Z = \sum_{ij} e_{ij} \otimes Z_{ij} \in M_N(\C)\otimes \mcO_q^{\br}(M_N(\C)) = M_N(\mcO_q^{\br}(M_N(\C))),
	\] 
	with universal relations given by the \emph{reflection equation}
	\begin{equation}\label{EqUnivRelZ}
		\hat{R}_{12}Z_{23} \hat{R}_{12}Z_{23} = Z_{23} \hat{R}_{12}Z_{23}\hat{R}_{12}. 
	\end{equation}

	We define the $*$-REA $\mcO_{q}(H(N))$ as $\mcO_{q}^{\br}(M_N(\C))$ endowed with the $*$-structure $Z^* = Z$. 
\end{Def*}

The reflection equation was introduced by Cherednik \cite{Ch84} motivated by the study of quantum scattering on the half-line. The reflection equation algebra is an algebraic approach to the study of solutions of the reflection equation, and has been studied in \cite{Maj91,KS92,KS93,KSS93,Maj94,PS95,DM02,Mud02,DKM03,Mud06,KS09,JW20}. There is a close relationship between the reflection equation and Yang-Baxter equation, which extends to the algebraic structure via quantum group coactions. In the case of $\mcO_q(H(N))$, there are natural coactions (i.e. comodule $*$-algebra structures) of the quantum groups $\mcO_q(U(N))$ and $U_q(\mfu(N))$. For a more detailed introduction to the reflection equation algebra, see \cite{DCMo24a}.

This is the final part of a series of papers \cite{DCMo24a,DCMo24b} aimed at achieving a classification of representations of the reflection equation algebra. Here by a representation of $\mcO_q(H(N))$ we mean a unital $*$-homomorphism $\pi$ into the algebra of bounded operators on a Hilbert space $\mcH$,
\[
\pi:\mcO_q(H(N))\rightarrow B(\mcH).\]

We call the representation \emph{irreducible} if $\mcH$ is the only non-trivial closed invariant subspace of $\pi$.

A key result of \cite{DCMo24a} was to show that $\mcO_q(H(N))$ is \emph{type I}. This means that its representations are in principle classifiable. Further, \cite{DCMo24a} developed two invariant properties of $\mcO_q(H(N))$ irreducibles that will be needed for the classification. The first invariant associated to an irreducible $\pi$ is known as the \emph{signature}
\[
\zeta(\pi)=(N_+,N_-,N_0)\in \N^3, \qquad N_++N_-+N_0=N.\]
The signature of a representation is preserved under the respective coactions of the quantum groups $\mcO_q(U(N))$ and $U_q(\mfu(N))$, and can be thought of as a quantization of Sylvester's law of inertia.

The second invariant described in \cite{DCMo24a} is the \emph{central $*$-character}. 
\begin{Def}
The centre of $\mcO_q(H(N))$ is generated by elements $\sigma_k$, $1\leq k\leq N$. A \emph{central $*$-character} describes the scalar values in $\C$ assigned to the centre by a given irreducible.
\end{Def}
In particular, \cite{DCMo24a} explicitly classified the possible values of a central $*$-character.

$\mcO_q(H(N))$ can be viewed as a deformation of the algebra of functions on the space of $N\times N$ Hermitian matrices. In \cite{DCMo24b}, the classical limit of this viewpoint was considered, and used to motivate the introduction of \emph{shapes} for the reflection equation algebra.
\begin{Def}
A (self-adjoint) \emph{shape matrix} $S$ is a self-adjoint matrix whose entries are all unimodular or zero, such that all eigenvalues of $S$ are in $\{1,-1,0\}$.
\end{Def}
A shape matrix can be assigned to each $\mcO_q(H(N))$ irreducible $\pi$, and be thought of as labelling a family of elements lying in $\Ker(\pi)$, while simultaneously describing a family of $q$-commuting operators in $\pi(\mcO_q(H(N)))$. The main result of \cite{DCMo24b} was to show that each irreducible $\mcO_q(H(N))$ representation has a (not necessarily self-adjoint) unique shape. That the shape matrix must be self-adjoint will be shown later in this paper in Corollary \ref{cor: self adjoint}.

We can assign a notion of signature to both central $*$-characters and shapes. Further, similar to the signature, the central $*$-character and shape also display some invariance properties under quantum group coactions: The central $*$-character is preserved by the $\mcO_q(U(N))$ coaction, but not by the $U_q(\mfu(N))$ coaction. Whilst the shape is preserved by the $U_q(\mfu(N))$ coaction but not by the $\mcO_q(U(N))$ coaction.

The main aim of the current paper is to develop a \emph{highest weight theory} on pre-Hilbert spaces based on the family of $q$-commuting operators assigned to a given shape. Further, we will see that the highest weights are closely related to the central $*$-characters. Using our highest weight theory, we will then obtain a classification of the bounded $*$-representations of $\mcO_q(H(N))$ by showing the following:
\begin{TheoIntro}
Every bounded irreducible $*$-representation of $\mcO_q(H(N))$ is uniquely classified by the pair
\[
(S,s),\]
where $S$ is a self-adjoint shape and $s$ is a central $*$-character of signature compatible with $S$.
\end{TheoIntro}

The contents of the paper are as follows. In Section \ref{section: shapes def} we review some algebraic properties of the reflection equation algebra, as well as the definition of shapes. In Section \ref{section: coaction def} we review properties of the $\mcO_q(U(N))$ and $U_q(\mfu(N))$ coactions, along with a special family of representations called big cell representations. We define our highest weight theory for the reflection equation algebra in Section \ref{section: highest weight theory}. In Section \ref{section: coaction on shapes}, we study how the $\mcO_q(U(N))$ coaction acts on shapes. This coaction is used in Section \ref{section: highest weights for shapes} to classify the possible highest weights for shapes. In Section \ref{section: fin dim reps}, the possible shapes of finite dimensional representations are classified. Finally in Section \ref{section: classification}, the results of the previous sections are combined to obtain the classification result.

We end this introduction with some common notations that we will need.

For $k\in\Z_{\geq 0}$, we write $\brk$ for the set $\{1,2,...,k\}$. We also write $\binom{\brN}{k}$ for the set of subsets $I\subseteq\brN$ with $\lvert I\rvert=k$. We will always view a set $I\subseteq\brN$ as a totally ordered set in the form
\[
I = \{i_1<...<i_k\}.\]
We will need to consider two orderings on the set of subsets $\binom{\brN}{k}$, a total order and partial order, denoted $I\leq J$ and $I\preceq J$ respectively:
\begin{itemize}
	\item $\leq$ is the lexicographic order given by
	\begin{equation}\label{EqOrdTot}
		I < J \qquad \textrm{if and only if}\qquad i_1\leq j_1,\ldots, i_{p-1}\leq j_{p-1},i_p < j_p\textrm{ for some }1\leq p \leq k.
	\end{equation}
	\item $\preceq$ is the partial order given by
	\begin{equation}\label{EqOrdPart}
		I \preceq J \qquad \textrm{if and only if}\qquad i_p \leq j_p\textrm{ for all }1\leq p \leq N.
	\end{equation}
\end{itemize}
Both of these orders extend lexicographically to orders on $\binom{\brN}{k}\times \binom{\brN}{k}$.

For $I\in \binom{\brN}{k}$ and $K\in \binom{\brk}{l}$ we define
\begin{equation}\label{EqSetUnderUpper}
	I_K = \{i_{p}\mid p \in K\},\qquad I^K = I \setminus I_K.
\end{equation}
We also denote by $I\Delta J$ the \emph{symmetric difference}, i.e.
\[
I\Delta J = (I\cup J)\setminus(I\cap J).\]
For $I \in \binom{\brN}{k}$, we call the sum of its elements the \emph{weight} of $I$, 
\begin{equation}\label{eq: wt(I) def}
\wt(I) = \sum_{r=1}^k i_r.
\end{equation}

Finally, given a bijection $\sigma: A \rightarrow B$ between finite totally ordered sets, we write
\begin{equation}\label{EqLengthSym}
	l(\sigma) = |\{(i,j) \in A\times A \mid i<j,\sigma(i)>\sigma(j)\}|
\end{equation}
for the number of inversions in $\sigma$.

\section{Shapes and properties of quantum minors}\label{section: shapes def}

The relations of $\mcO_q(H(N))$ defined from \eqref{EqUnivRelZ} can be written out explicitly in terms of generators as follows
\begin{multline}\label{EqCommOqHN}
	q^{-\delta_{ik} -\delta_{jk}} Z_{ij} Z_{kl} + \delta_{k<i}(q^{-1}-q) q^{-\delta_{ij}} Z_{kj}Z_{il} +\delta_{jk} (q^{-1}-q)q^{-\delta_{ij}}\sum_{p<j} Z_{ip} Z_{pl} + \delta_{ij} \delta_{k<i}(q^{-1}-q)^2  \sum_{p<i} Z_{kp}Z_{pl} \\
	= q^{-\delta_{il}-\delta_{jl}} Z_{kl} Z_{ij} + \delta_{l<j}(q^{-1}-q)q^{-\delta_{ij}}  Z_{kj}Z_{il} + \delta_{il}(q^{-1}-q) q^{-\delta_{ij}} \sum_{p<i} Z_{kp}Z_{pj} +\delta_{ij}\delta_{l<j} (q^{-1}-q)^2  \sum_{p<j} Z_{kp}Z_{pl}. 
\end{multline}
In this form, the $\mcO_q(H(N))$ relations are generally unwieldy and difficult to work with. For our purposes, it will be much more productive to proceed using \emph{quantum minors}. These are certain elements in $\mcO_q(H(N))$ that can be thought of as deformations of sub-determinants. We refer to \cite{DCMo24a} for detailed proofs and references for any results concerning quantum minors that follow.

Recall from Section 2.4 of \cite{DCMo24a} the terms $\hat{R}^{IJ}_{I'J'}\in\C$ for $I,J\in \binom{[N]}{k}$ and $I',J'\in \binom{[N]}{l}$. They were defined by
\[
\hat{R}^{IJ}_{I'J'} := \mbr(X_{JI},X_{I'J'}),\]
where $\mbr$ is the coquasitriangular structure on $\mcO_q(U(N))$. The elements $X_{JI}$, $X_{I'J'}$ are quantum minors in $\mcO_q(U(N))$, or alternatively the matrix coefficients of $\mcO_q(U(N))$ irreducible comodules. The quantum minors in $\mcO_q(H(N))$ we want to consider were defined in \cite{DCMo24a} as 'covariantised' versions of the $\mcO_q(U(N))$ quantum minors. 

The terms satisfy 
\begin{equation}\label{EqCondR}
	\hat{R}^{IJ}_{I'J'} \neq 0 \quad \textrm{only if}\quad  J\preceq I,J'\preceq I'\quad \textrm{and}\quad J \setminus I = J'\setminus I',\;I\setminus J = I'\setminus J',
\end{equation}
and
\[
\hat{R}^{II}_{I'I'} = q^{- |I\cap I'|}. 
\] 
There are also inverse terms $(\hat{R}^{-1})_{I'J'}^{IJ}$ that again satisfy the conditions of equation \eqref{EqCondR}, however in the inverse case we have $(\hat{R}^{-1})^{II}_{I'I'} = q^{|I\cap I'|}$. 

The $\mcO_q(H(N))$ quantum minors can be constructed inductively using Laplace expansion formulae:
\begin{Prop}
	For $I,J \in \binom{\brN}{k}$ and $K\in \binom{\brk}{m}$ with $m\leq k \leq N$ we have
	\begin{equation}\label{EqLaplExp1}
		Z_{I,J}= \sum_{P \in \binom{\brk}{m}}\sum_{S,T\in \binom{\brN}{m}}\sum_{S',T'\in \binom{\brN}{k-m}}(-q)^{\wt(P) -\wt(K)} (\hat{R}^{-1})^{S,I_K}_{I^K,T'}  \hat{R}^{J_P,T}_{T',S'} Z_{S,T} Z_{S',J^P},
	\end{equation}
	\begin{equation}\label{EqLaplExp2}
		Z_{I,J} =  \sum_{P \in \binom{\brk}{m}}\sum_{S,T\in \binom{\brN}{m}}\sum_{S',T'\in \binom{\brN}{k-m}}(-q)^{\wt(P) -\wt(K)} (\hat{R}^{-1})^{T,J_K}_{J^K,S'}  \hat{R}^{I_P,S}_{S',T'} Z_{I^P,T'} Z_{S,T},
	\end{equation}
	\begin{equation}\label{EqLaplExp3}
		Z_{I,J} =  \sum_{P \in \binom{\brk}{m}}\sum_{S,T\in \binom{\brN}{m}}\sum_{S',T'\in \binom{\brN}{k-m}}(-q)^{\wt(P) -\wt(K)} (\hat{R}^{-1})^{S,I_P}_{I^P,T'}  \hat{R}^{J_K,T}_{T',S'} Z_{S,T} Z_{S',J^K},
	\end{equation}
and
	\begin{equation}\label{EqLaplExp4}
	Z_{I,J} =  \sum_{P \in \binom{\brk}{m}}\sum_{S,T\in \binom{\brN}{m}}\sum_{S',T'\in \binom{\brN}{k-m}}(-q)^{\wt(P) -\wt(K)} (\hat{R}^{-1})^{T,J_P}_{J^P,S'}  \hat{R}^{I_K,S}_{S',T'} Z_{I^K,T'} Z_{S,T}.
\end{equation}	
\end{Prop}
\begin{proof}
The first two expansions were proven in \cite{DCMo24a}. The third expansion can be similarly shown by combining equations $(2.18)$ and $(2.25)$ of \cite{DCMo24a}. By applying $*$ to \eqref{EqLaplExp3}, we get then the fourth expansion.
\end{proof}

The quantum minors satisfy a general commutation relation: For all $I,J,I',J'$, we have
\begin{equation}\label{EqGenCommRel}
	\sum_{K,L,L'} \left(\sum_{P'}\hat{R}_{JK}^{P'I'}\hat{R}_{P'L'}^{IL}\right)  Z_{K,L} Z_{L',J'} =   \sum_{K,L,L'}  \left(\sum_{P'}\hat{R}_{JK}^{P'L'}\hat{R}_{P'J'}^{IL}\right) Z_{I',L'} Z_{K,L}.
\end{equation}

A key relation we will need concerning quantum minors, is a result known as the `common-submatrix Laplace expansion', and attributed to Muir \cite{Mui60,Wil86,Wil15}. We use the form given in Proposition 2.6 of \cite{DCMo24b}:
\begin{Prop}
	Let $I,J \in \binom{\brN}{k}$ and $F,G\in \binom{[k]}{k-r}$. Then for any $K,K' \in \binom{\brr}{l}$ we have 
	\begin{eqnarray} 
		& \hspace{-8cm}\delta_{K,K'}\sum_{S,T,H,L} (\widehat{R}^{-1})^{S,I}_{I_F,H} \widehat{R}^{J,T}_{H,L} Z_{S,T}Z_{L,J_G} \nonumber \\ 
		&= \sum_{A,B,C,D} \sum_{P \in \binom{\brr}{l}} (-q)^{\wt(P) -\wt(K)} (\widehat{R}^{-1})^{A,I_F\cup (I^F)_K}_{I_F\cup (I^F)^{K'},B} \widehat{R}^{J_G\cup (J^G)_P,C}_{B,D} Z_{A,C} Z_{D,J_G\cup (J^G)^P}  \label{EqMuirBr} \\
		&=  \sum_{A,B,C,D} \sum_{P \in \binom{\brr}{l}} (-q)^{\wt(P) -\wt(K)} (\widehat{R}^{-1})^{A,I_F\cup (I^F)_P}_{I_F\cup (I^F)^{P},B} \widehat{R}^{J_G\cup (J^G)_K,C}_{B,D} Z_{A,C} Z_{D,J_G\cup (J^G)^{K'}}. \label{EqMuirBr2} 
	\end{eqnarray}
\end{Prop}

\subsection{The shape and rank of a representation}

The shape of an $\mcO_q(H(N))$ representation was introduced in \cite{DCMo24b} motivated by the quasi-classical limit of the corresponding Poisson structure. 
\begin{Def}\label{def:shape}
A \emph{shape} $S$ is a pair $S=(\tau,u)$, where
\begin{itemize}
	\item $\tau$ is a permutation of $\{1,2,...,N\}$, and
	\item $u$ is a map $u:\{1,2,...,N\}\rightarrow\{0\}\cup\T$ such that $u_{i}\neq 0$ if $i\neq \tau(i)$.
\end{itemize}
We call $P=\{p\in[N]\vert u_p\neq 0\}$ the \emph{associated set}, and $M=\lvert P\rvert$ the \emph{rank} of $S$. A shape is called \emph{self-adjoint} if $\tau$ is an involution and $u_{\tau(i)}=\bar{u}_i$ for all $i$.
\end{Def}
A shape $S$ can be encoded by a matrix $S\in M_N(\{0\}\cup\T)$ via $Se_i=u_ie_{\tau(i)}$. It follows that $S$ is self-adjoint if and only if the corresponding matrix is self-adjoint. Further the rank of $S$ is equal to the rank of the matrix. In general, we will consider a shape and its corresponding matrix as equivalent.

Given a shape $S$ with associated set $P$, denote $P=\{p_1<p_2<...<p_M\}$. Then we also denote
\[
P_{[k]}=\{p_1,p_2,...,p_k\}, \qquad \tau(P_{[k]}) = \{\tau(p_1),\tau(p_2),...,\tau(p_k)\}.\]

\begin{Def}
For $1\leq M\leq N$, denote by $I_M$ the 2-sided $*$-ideal of $\mcO_q(H(N))$ generated by all $Z_{IJ}$ with $I,J\in\binom{\brN}{M}$. We say an irreducible representation $\pi$ of $\mcO_q(H(N))$ has \emph{rank $M$} if $I_{M+1}\subseteq\Ker\pi$ but $I_M\not\subseteq\Ker\pi$.	
\end{Def}

\begin{Def}
Given a shape $S$ of rank $M$, we define $I^{\leq}_S$ to be the 2-sided $*$-ideal generated by $I_{M+1}$, and for all $k\leq M$, the $Z_{IJ}$ with $I,J\in\binom{\brN}{k}$ such that $(J,I)<(P_{[k]},\tau(P_{[k]}))$. 
\end{Def}

Given a shape $S$ of rank $M$, we will commonly denote 
\[
Z_{S,k}:=Z_{\tau(P_{[k]}),P_{[k]}} \text{  for  } 1\leq k\leq M. \]

\begin{Lem}\label{LemqCommpi}
In $\mcO_q(H(N))/I^{\leq}_S$, we have the following commutation relations
	\begin{equation}\label{EqqCommQuot}
		Z_{S,k} Z_{I,J} = q^{|I\cap P_{\brk}| + |I \cap \tau(P_{\brk})| - |J\cap P_{\brk}| - |J\cap \tau(P_{\brk})|} Z_{I,J}Z_{S,k}.
	\end{equation}
\end{Lem}

\begin{Def}
Given a shape $S=(\tau,u)$, we say an irreducible representation $\pi$ of $\mcO_q(H(N))$ has shape $S$ if
\begin{itemize}
	\item $I^{\leq}_S\subseteq \Ker\pi$,
	\item $\Ker(\pi(Z_{S,k}))=0$ for all $k\leq M$, and
	\item $(-1)^{l(\tau_{\mid P_{[k]}})}\overline{u_{P_{[k]}}}\pi(Z_{S,k})$ is a positive operator for all $k\leq M$.
\end{itemize}	
\end{Def} 	

The main result of \cite{DCMo24b} was the following:
\begin{Theorem}\label{theorem: unique shape}
Every irreducible representation $\pi$ of $\mcO_q(H(N))$ has a unique shape $S$. Further, $\pi$ has rank $M$ if and only if $S$ has rank $M$. 
\end{Theorem}
We will extend this result by showing that the shape of any irreducible representation must be self-adjoint in Corollary \ref{cor: self adjoint}.

We finish this section by noting some other algebraic properties of $\mcO_q(H(N))$ that we will need:

\begin{Lem}\label{Lem OqHn quotient} Let $J_M$ be the 2-sided $*$-ideal generated by the $Z_{ij}$ with $1\leq i \leq M$ and $1\leq j \leq N$. Then 
	\[
	\mcO_q(H(N))/J_M \cong \mcO_q(H(N-M)),\qquad (Z_{ij})_{1\leq i,j\leq N} \mapsto \begin{pmatrix} 0_M & 0_{M,N-M} \\ 0_{N-M,M} & (Z_{ij})_{1\leq i,j\leq N-M}\end{pmatrix}
	\]
	is a $*$-homomorphism. 
\end{Lem}

The centre of $\mcO_q(H(N))$ is a polynomial algebra over $N$ generators, denoted $\sigma_1,...,\sigma_N$. An explicit form for these generators was given in \cite{JW20}, see also \cite{Flo20}.

\begin{Theorem}\label{TheoCentrPol}
	The centre of $\mcO_q(H(N))$ is generated by the self-adjoint algebraically independent elements $\sigma_1,...,\sigma_N$ given by
	\[
	\sigma_k =  q^{2Nk} \sum_{I\in \binom{\brN}{k}} \sum_{\sigma \in \Sym(I)}q^{-2\wt(I)} (-q)^{-l(\sigma)} q^{-a(\sigma)} Z_{i_k\sigma(i_k)}\ldots Z_{i_1 \sigma(i_1)},
	\] 
	where $\Sym(I)$ is the group of permutations of $\brN$ fixing $\brN\setminus I$, where $wt(I)$ and $l(\sigma)$ are defined by \eqref{eq: wt(I) def} and \eqref{EqLengthSym} respectively, and where $a(\sigma) = |\{l\mid \sigma(l)<l\}|$.
\end{Theorem}
We will often call $\sigma_1$ and $\sigma_N$ the \emph{quantum trace} and the \emph{quantum determinant} respectively. The quantum determinant is related to the quantum minors by
\[
Z_{[N]} := Z_{[N],[N]} = q^{-N(N-1)}\sigma_N.\]

Finally we recall from \cite{Mud02} the classification of $*$-characters
\[
\chi:\mcO_q(H(N))\rightarrow \C.\]

\begin{Theorem}\label{Character classification theorem}
	Each $*$-character $\chi$ of $\mcO_q(H(N))$ is of the form $\chi = \chi_{k,l,a,c,y}$ where 
	\begin{itemize}
		\item $k,l\in \Z_{\geq0}$ with $k+l \leq N-l$,
		\item $a>0$ and $c\in \R^{\times}$, 
		\item $y = (y_{0},y_{1},\ldots,y_{l-1})$ with $|y_i| = 1$, 
	\end{itemize}
	such that
	\[
	\chi(Z) = c\left(a \sum_{i = k+l+1}^{N} e_{ii} - a^{-1} \sum_{i=N-l+1}^N e_{ii} +   \sum_{i=0}^{l-1} (y_{i} e_{k+i+1,N-i} + \overline{y_{i}} e_{N-i,k+i+1})\right).
	\] 
\end{Theorem}

\section{The $\mcO_q(U(N))$ and $\mcO_q(T(N))$ comodule $*$-algebra structure, and big cell representations}\label{section: coaction def}

To classify the representations of $\mcO_q(H(N))$, we will need to use not just the algebraic structure of $\mcO_q(H(N))$, but also its coactions via comodule $*$-algebra structures over $\mcO_q(U(N))$ and $\mcO_q(T(N))$. The $\mcO_q(T(N))$ coaction in particular can be used to describe an essential family of $\mcO_q(H(N))$ representations previously studied in \cite{DCMo24a} known as \emph{big cell representations}.

\subsection{The $\mcO_q(H(N))$ comodule $*$-algebra over $\mcO_q(U(N))$}

Let $X = \sum\limits_{i,j}^{N}e_{ij}\otimes X_{ij}$ for variables $X_{ij}$, $1\leq i,j\leq N$. 

\begin{Def}\cite{FRT88}. 
	The FRT algebra (= Faddeev-Reshetikhin-Takhtajan) $\mcO_q(M_N(\C))$ is the unital $\C$-algebra defined by the universal relations
	\begin{equation}
		\hat{R}_{12}X_{13}X_{23} = X_{13}X_{23}\hat{R}_{12}.
	\end{equation}
	It has a bialgebra structure given by
	\begin{equation}
		(\id\otimes\Delta)X = X_{12}X_{13}, \qquad (\id\otimes\varepsilon)X = I_N.
	\end{equation}
\end{Def}
$\mcO_q(M_N(\C))$ has a unique central element \cite[Theorem 4.6.1]{PW91} given by
\begin{equation}\label{EqQuantDet}
	\Det_q(X) := X_{[N]} = X_{[N],[N]} = \sum_{\sigma \in S_N} (-q)^{l(\sigma)} X_{1,\sigma(1)}X_{2,\sigma(2)}\cdots X_{N,\sigma(N)}.
\end{equation}
By considering the localisation $\mcO_q(M_N(\C))[\Det_q(X)^{-1}]$, $X$ becomes invertible, and we can define a $*$-structure via
\begin{equation}
	X^*:=X^{-1}.
\end{equation} 
Under this consideration, $X$ now defines a Hopf $*$-algebra, which we denote by $\mcO_q(U(N))$. 

\begin{Prop}\label{prop: UN coaction}
	There is a coaction on $\mcO_q(H(N))$ given by 
	\begin{equation}
		\Ad^{U}_q: \mcO_q(H(N)) \rightarrow \mcO_q(H(N)) \otimes \mcO_q(U(N)),\qquad Z \mapsto X_{13}^{*}Z_{12}X_{13}.
	\end{equation}
	With respect to the generating elements, the coaction is given by 
	\begin{equation}
		\Ad^{U}_q: Z_{ij}\mapsto \sum\limits_{k,l}Z_{kl}\otimes X_{ki}^{*}X_{lj},
	\end{equation}
	further the central elements are coinvariant, i.e.
	\begin{equation}
		\Ad^{U}_q: \sigma_{i}\mapsto \sigma_{i}\otimes 1.
	\end{equation}
\end{Prop}
For proof see for example \cite{DCF19}.

From \cite[Proposition 10.30]{KS97} or \cite[Theorem 7.4.1]{Maj95} we can further describe the coaction in terms of quantum minors
\begin{Prop}
In terms of $\mcO_q(H(N))$ quantum minors, the $\mcO_q(U(N))$ coaction is given by
\begin{equation}\label{equation: qminor coaction}
\Ad^{U}_q: Z_{IJ}\mapsto\sum\limits_{K,L} Z_{KL}\otimes X_{KI}^{*}X_{LJ}.
\end{equation}
\end{Prop}

For our purposes, we will generally not be considering the $\mcO_q(U(N))$ coaction directly, but will instead be combining it with certain $\mcO_q(U(N))$ representations.

Denote by $\mcO_q(SU(2))$ the quotient of $\mcO_q(U(2))$ given by $\Det_q(X)\mapsto 1$. Then the generating matrix of $\mcO_q(SU(2))$ has the form 
\[
X = \begin{pmatrix}
	X_{11} & -qX_{21}^{*} \\ X_{21} & X_{11}^{*}
\end{pmatrix}\]
Every $\mcO_q(SU(2))$ irreducible is of the form $\mbs\otimes\chi_\theta$, where $\chi_{\theta}$ is the $*$-character
\[
\chi_{\theta}(X) = \begin{pmatrix}
	e^{2\pi i\theta} & 0\\ 0 & 	e^{-2\pi i \theta}
\end{pmatrix}, \qquad \theta\in\R/\Z,\]
and $\mbs$ is the $l^{2}(\N)$ representation defined by
\begin{equation}
X_{11}e_i = \sqrt{1-q^{2i}}e_{i-1}, \qquad X_{21}e_{i} = q^i e_i.
\end{equation}

There are a family of quotient maps 
\[
\res_i:\mcO_q(U(N))\rightarrow\mcO_q(SU(2))\] given by
\[
\res_i:\begin{pmatrix}
	X_{ii} & X_{i,i+1}\\ X_{i+1,i} & X_{ii}\end{pmatrix}\mapsto\begin{pmatrix}
	X_{11} & X_{12} \\ X_{21} & X_{22}\end{pmatrix}, \qquad X_{kl}\mapsto\delta_{kl}\text{  otherwise.}\]
We denote by $\mbs_i$ the representation given by the composition 
\[
\mbs\circ\res_i:\mcO_q(U(N))\rightarrow B(l^{2}(\N)).\] 
This representation is again irreducible, and every irreducible $\mcO_q(U(N))$ representation can be constructed as a tensor product of such representations:

\begin{Theorem}\cite{Koe91,LS91}.
Every irreducible representation of $\mcO_q(U(N))$ is of the form 
\[
\mbs_w\otimes\chi_{\theta},\]
where $\chi_{\theta}$ is an $\mcO_q(U(N))$ $*$-character, and $\mbs_w=\mbs_{i_{1}}\otimes...\otimes\mbs_{i_{k}}$, for $w=s_{i_{1}}...s_{i_{k}}$ a reduced word in $\Sym(\brN)$ transpositions $s_{i}=(i,i+1)$. Further, $\mbs_w\simeq \mbs_v$ if and only if $w$ and $v$ are equal as reduced words in $\Sym(\brN)$.
\end{Theorem}

\begin{Def}
We denote by $\alpha_{i}$ the coaction given by the composition
\begin{equation}\label{def: coaction UN}
(\id\otimes \mbs_i)\circ\Ad^{U}_q:\mcO_q(H(N))\rightarrow\mcO_q(H(N))\otimes B(l^2(\N)).
\end{equation}
\end{Def}

\subsection{The $\mcO_q(H(N))$ comodule $*$-algebra over $\mcO_q^\epsilon(T(N))$}

Choose $\epsilon\in\{1,-1,0\}^N$, and denote 
\[
\epsilon_{(i,j]}:=\prod\limits_{k=i+1}^{j}\epsilon_{k}.\]
We will always assume that $\epsilon$ is in a \emph{standard form}. By this, we mean we assume that $\epsilon$ consists of $M$ non-zero terms followed by $N-M$ zero terms. We refer to $M$ as the \emph{rank of $\epsilon$}.

We define the $\epsilon$-deformed braid operator
\[
\hat{R}_\epsilon= \sum\limits_{i,j}q^{-\delta_{ij}}e_{ji}\otimes e_{ij}+(q^{-1}-q)\sum\limits_{i<j}\epsilon_{(i,j]}e_{jj}\otimes e_{ii}.\]
Let $T=\sum\limits_{i,j}^{N}e_{ij}\otimes T_{ij}$, where we put $T_{ij}=0$ for $i>j$. Let $T^*$ be the formal adjoint of $T$, where we further identify $T_{ii}^{*}=T_{ii}$.
\begin{Def}
The $*$-algebra $\mcO_q^\epsilon(T(N))$ is given by the universal relations
\begin{equation}
\hat{R}_{12}T_{13}T_{23} = T_{13}T_{23}\hat{R}_{12}, \qquad T_{23}^*T_{13}^*\hat{R}_{12} = \hat{R}_{12}T_{23}^*T_{13}^*, \qquad T_{23}\hat{R}_{12}T_{23}^* = T_{13}^*\hat{R}_{\epsilon,12}T_{13}.
\end{equation}
We denote the case $\epsilon=\{1,1,...,1\}$ by $\mcO_q(T(N))$.
\end{Def}
The localisation of $\mcO_q(T(N))$ forms a Hopf algebra with coproduct
\[
(\id\otimes\Delta)T = T_{12}T_{13}.\]
For other $\epsilon$, $\mcO_q^\epsilon(T(N))$ forms a comodule algebra over $\mcO_q(T(N))$. For details of the following, see for example \cite[Theorem 8.33]{KS97}:
\begin{Prop}
The localisation of $\mcO_q(T(N))$ by $T_{ii}^{-1}$, $1\leq i\leq N$, is isomorphic as a Hopf $*$-algebra to $U_q(\mfu(N))^{\cop}$, with the isomorphism defined by
\[
T_{ii}\mapsto K_i^{-1}, \qquad T_{i,i+1}\mapsto (q^{-1}-q)F_iK_{i+1}^{-1}.\]
\end{Prop} 
Similar to the $\mcO_q(U(N))$ case, we have an $\mcO_q(H(N))$ coaction given by
\begin{equation}
\Ad^T_q:\mcO_q(H(N))\rightarrow \mcO_q(H(N))\otimes\mcO_q(T(N)), \qquad Z\mapsto T_{13}^*Z_{12}T_{13}.
\end{equation}
In terms of quantum minors, we have
\begin{equation}\label{equation: T coaction q minors}
	\Ad^T_q:Z_{IJ}\mapsto\sum\limits_{K,L}Z_{KL}\otimes T_{KI}^*T_{LJ}.
\end{equation}
In particular, for the leading minors $Z_{[k]} = Z_{[k],[k]}$ we have
\begin{equation}
\Ad^T_q:Z_{[k]} \mapsto Z_{[k]}\otimes\prod\limits_{1\leq i\leq k}T_{ii}^{2}.
\end{equation}	
Considering $\Ad^T_q$ with $Z$ replaced by the identity matrix, we get the following well-known result \cite[Theorem 3]{Bau00}
\begin{Prop}
There is an injective $*$-homomorphism 
\[
i_T:\mcO_q(H(N))\rightarrow\mcO_q(T(N)), \qquad Z\mapsto T^*T.\]
\end{Prop}
In particular, the localisation of $\mcO_q(H(N))$ by the leading quantum minors $Z_{[k]}$ embeds into $U_q(\mfu(N))^{\cop}$.

Denote by $1_\epsilon$ the $N\times N$ diagonal matrix with entries
\[
\epsilon_{(0,1]}, \epsilon_{(0,2]},...,\epsilon_{(0,N]}.\]
Then there is a natural generalisation of the previous homomorphism
\begin{Prop}\label{prop: TN embedding}
Let $\epsilon\in\{1,-1,0\}^N$ be in standard form of rank $M$. Then there is an injective $*$-homomorphism
\[
i^{\epsilon}_T:\left(\mcO_q(H(N))/I_{M+1}\right)[Z_{[1]}^{-1},...,Z_{[M]}^{-1}]\rightarrow\mcO_q^\epsilon(T(N)), \qquad Z\mapsto T^*1_\epsilon T,\]
which maps
\begin{equation}
i^{\epsilon}_T:Z_{[k]}\mapsto \prod\limits_{i=1}^{k}\epsilon_{(0,i]}T_{ii}^{2}, \qquad 1\leq k\leq M.
\end{equation}
\end{Prop}

\begin{Def}\label{def: TN highest weight}
Let $V$ be an $\mcO_q^\epsilon(T(N))$ representation. We call $v\in V$ a \emph{weight vector} if it is a joint eigenvector for all $T_{ii}\in\mcO_q^\epsilon(T(N))$. We call $v_{0}\in V$ a \emph{highest weight vector} if it is a weight vector and $T_{ij}^*v_{0}=0$ for all $i<j$. The set of $T_{ii}$ eigenvalues associated to a given (highest) weight vector are called the \emph{(highest) weights}. We call $V$ a \emph{unitarizable highest weight representation} if $V$ is generated by a unique highest weight vector, and $V$ is a pre-Hilbert space with respect to the $\mcO_q^\epsilon(T(N))$ $*$-structure.
\end{Def}
The elements $T_{ii}$ will act on a highest weight vector as
\[
T_{ii}v_{0} = q^{r_{i}}v_{0}, \qquad r\in\R.\]
We will often simplify somewhat and refer to the set $r\in\R^M$ as the highest weight instead of $\{q^{r_{1}},...,q^{r_{M}}\}$.
 
\begin{Def}
For a fixed choice of $\epsilon$ in standard form of rank $M$, we call $r\in\R^M$ $\epsilon$-adapted when for all $1\leq s<t\leq M$ with $\epsilon_{(s,t]}=1$, we have
\begin{equation}
(r_t+t)-(r_s+s)\in\Z_{>0}.
\end{equation} 
\end{Def}
The unitarizable highest weight representations of $\mcO_q^\epsilon(T(N))$ can then be described as follows:
\begin{Theorem}\label{theorem: highest weights of TN}
For a fixed choice of $\epsilon$ of rank $M$, $r\in\R^M$ is the highest weight of a unitarizable highest weight $\mcO_q^\epsilon(T(N))$ representation if and only if $r$ is $\epsilon$-adapted.
\end{Theorem}

Via the map from Proposition \ref{prop: TN embedding}, every unitarizable highest weight $\mcO_q^\epsilon(T(N))$ representation becomes an irreducible $\mcO_q(H(N))$ $*$-representation. We call the representations obtained this way \emph{big cell representations}. We can also give an alternative description of big cell representations in terms of shapes:
\begin{Def}
An irreducible $\mcO_q(H(N))$ representation of rank $M$ is called a \emph{big cell representation} if it has shape given by $\tau=\id$ and $u(i)=\epsilon_{(0,i]}$ for a fixed choice of $\epsilon\in\{1,-1,0\}^N$ in standard form of rank $M$. We sometimes refer to this as a big cell representation of \emph{shape $\epsilon$}.
\end{Def}
Theorem \ref{theorem: highest weights of TN} can then be restated as follows:
\begin{Theorem}\label{theorem: highest weights of big cell reps}
$r\in\R^M$ is the highest weight of a big cell $\mcO_q(H(N))$ of shape $\epsilon$ if and only if $r$ is $\epsilon$-adapted.
\end{Theorem}
Using a form of Harish-Chandra homomorphism, explicit formulae for the central values of a big cell $\mcO_q(H(N))$ irreducible in terms of the $\mcO_q^\epsilon(T(N))$ highest weights can be given as follows:
\begin{Prop}\label{PropHC}
	Let $V$ be an irreducible $\mcO_q(H(N))$ big cell representation coming from an $\mcO_q^\epsilon(T(N))$ representation. Then with $e_k$ the $k$-th elementary symmetric polynomial in $N$ variables, we have the scalar coefficients of the central elements $\sigma_k$ acting on $V$ are given by
	\begin{equation}\label{EqHC}
		\sigma_k =  e_k(\epsilon_{(0,1]}T_{11}^2,\epsilon_{(0,2]}q^2T_{22}^2,\ldots, \epsilon_{(0,N]}q^{2N-2}T_{NN}^2). 
	\end{equation}
	In particular, for $r$ an $\epsilon$-adapted highest weight, we have
		\begin{equation}\label{EqHC2}
		\sigma_k =  e_k(\epsilon_{(0,1]}q^{2r_{1}},\epsilon_{(0,2]}q^{2r_{2}+2},\ldots, \epsilon_{(0,N]}q^{2r_{N}+2N-2}). 
	\end{equation}
\end{Prop}

In \cite{DCMo24a} it was shown that every $\mcO_q(H(N))$ irreducible is \emph{weakly contained} in a big cell representation. In other words, every irreducible factors through the C$^*$-envelope of a big cell representation. As a consequence, the scalar values of the central elements acting on an arbitrary $\mcO_q(H(N))$ irreducible must be equal to the scalar values given in \ref{PropHC} for some big cell representation.
\begin{Def}
Let $\msZ(\mcO_q(H(N)))$ denote the centre of $\mcO_q(H(N))$. Then for $\pi$ a big cell representation on $V$, and any $v\in V$, we call the corresponding map
\[
s_{\pi}:\msZ(\mcO_q(H(N)))\rightarrow\C, \qquad s_{\pi}:\sigma_k\mapsto\C, \qquad s_{\pi}(\sigma_k)v:=\pi(\sigma_k)v\]
a \emph{central $*$-character}.
\end{Def}

\begin{Def}\label{def: central char signature}
If $\pi$ is a rank $M$ irreducible big cell representation of $\mcO_q(H(N))$ of shape $\epsilon$, then we say $\pi$ has \emph{signature}
\[
\zeta(\pi)=\{N_+,N_-,M\}, \qquad N_++N_-+M=N,\]
if the set
\[
\{\epsilon_{(0,1]},\epsilon_{(0,2]},...\epsilon_{(0,N]}\}\]
consists of $N_+$ 1's, $N_-$ -1's, and $M$ 0's. If $\pi$ has central $*$-character $s$, then we say $s$ has signature $\zeta(\pi)$.
\end{Def}

\section{A highest weight theory for irreducibles}\label{section: highest weight theory}

Lemma \ref{LemqCommpi} shows that for an irreducible $\mcO_q(H(N))$-representation of fixed shape, each of the $Z_{S,k}$ will be diagonalizable with spectrum of the form $\{c_{k}q^{i}\}$. It is natural then to ask if we can develop a highest weight theory for irreducibles, using the $Z_{S,k}$ as an analogue of a Cartan subalgebra. One complication we need to overcome is that for general shapes, the commutation relations between the $Z_{s,k}$ and $Z_{ij}$ do not always give an obvious choice of lowering operators. Our aim is to define a choice of highest weight vector for general shapes such that we can use the $\mcO_{q}(U(N))$ coaction to relate our defined highest weight vector to the highest weight vector of a big cell representation. To simplify things for our highest weight theory, from now on we define an \textit{irreducible} $\mcO_q(H(N))$ representation to mean a bounded $*$-representation on a pre-Hilbert space $V$ with no non-trivial invariant subspaces. 

Before we proceed further we will first need to discuss what happens to a given shape when we restrict it to a smaller $\mcO_q(H(N))$. We will denote by $\downarrow V$ the restriction of an $\mcO_q(H(N))$-representation to $\mcO_q(H(N-1))$, and use similar notation for the associated data, e.g. if a representation $(V,\pi)$ has rank $M$, then $\downarrow M$ denotes the rank of $(\downarrow V,\downarrow \pi)$. When restricting to $\mcO_q(H(N-m))$, we use the notation $\downarrow^m V$. Given an $\mcO_{q}(H(N))$ shape, we use $Z_{\downarrow S,k}$, $Z_{\downarrow^{m}S,k}$ to denote the minors corresponding to the shape restricted to $\mcO_{q}(H(N-1))$, $\mcO_{q}(H(N-m))$ respectively.

To describe the restriction of shapes, we need the following lemma:
\begin{Lem}
If $S$ is a shape of rank $M$, and $1\leq k<M$ is such that there is a single $i\in P_{[k]}$ with $\tau(i)\notin P_{[k]}$, then we have
\begin{equation}\label{equation: weight relation 1}
	Z_{P_{[k]}\cup\{\tau(i)\}}Z_{P_{[k]}\setminus\{i\}}=-q^{-2}Z_{S,k}Z_{S,k}^{*}
\end{equation}
\end{Lem}
\begin{proof}
We first note that we will have $\tau(i)>j$ for all $j\in P_{[k]}$. Let $I=J=P_{[k]}\cup\{\tau(i)\}$, $F=G$ such that $I_F=P_{[k]}\setminus\{i\}$ and $I^F=\{i,\tau(i)\}$, and $K=K'=2$. Then consider equation \eqref{EqMuirBr}. From the conditions in \eqref{EqCondR}, the left hand side reduces to
\[
\sum\limits_{L}(\hat{R}^{-1})^{I,I}_{I_F,I_F}\hat{R}^{I,L\cup\{i,\tau(i)\}}_{P_{[k]}\setminus \{i\},L}Z_{I,L\cup\{i,\tau(i)\}}Z_{L,P_{[k]}\setminus \{i\}}.\]
The set $L$ must satisfy $L\preceq P_{[k]}\setminus \{i\}$, however we want to show that the only possibility is equality. Consider first the case $i=p_k$. Then we would have $P_{[k]}\setminus\{i\}=P_{[k-1]}$, and so if $L\prec P_{[k]}\setminus\{i\}$, we get
\[
(L,P_{[k]}\setminus\{i\})<(P_{[k-1]},\tau(P_{[k-1]})),\]
which gives zero. If $i<p_k$ and $L\prec P_{[k]\setminus\{i\}}$, then we would have $L\cup\{i\}\prec P_{[k]}$, and in particular $L\cup\{i,\tau(i)\}<P_{[k+1]}$. Hence we get
\[
(L\cup\{i,\tau(i)\},P_{[k]}\cup\{\tau(i)\})<(P_{[k+1]},\tau(P_{[k+1]})),\]
which again gives zero. Therefore the only possibility is $L=P_{[k]}\setminus\{i\}$, and so the left hand side reduces to
\[
Z_{P_{[k]}\cup\{\tau(i)\}}Z_{P_{[k]}\setminus \{i\}}.\] 
For the right hand side of equation \eqref{EqMuirBr}, again using the conditions in \eqref{EqCondR}, we see it reduces to
\[
\sum\limits_{P}(-q)^{\wt(P)-\wt(K)}(\hat{R}^{-1})^{\tau(P_{[k]}),\tau(P_{[k]})}_{P_{[k]},P_{[k]}}\hat{R}^{(P_{[k]}\setminus\{i\})\cup\{i,\tau(i)\}_P,C}_{P_{[k]},P_{[k]}}Z_{\tau(P_{[k]}),C}Z_{P_{[k]},(P_{[k]}\setminus\{i\})\cup\{i,\tau(i)\}^P},\]
with $C=(P_{[k]}\setminus\{i\})\cup\{i,\tau(i)\}_P$. The only non-zero choice is $P=1$, and so simplifying we get
\[
-q^{-2}Z_{\tau(P_{[k]}),P_{[k]}}Z_{P_{[k]},\tau(P_{[k]})}.\]
\end{proof}

\begin{Prop}\label{proposition: restriction of shapes}
	Given an $\mcO_q(H(N))$-representation $V$ of shape $S= (S_{ij})_{1\leq i,j\leq N}\simeq(\tau,u)$, the shape $\downarrow S$ of $\downarrow V$ is well-defined, and given by $\downarrow S = (S_{ij})_{1\leq i,j\leq N-1}$. Further, the permutation $\tau$ must be an involution. 	
\end{Prop}	

\begin{proof}
	The restriction of the shape only depends on the restriction of diagonal and zero operators, so can be seen to be well-defined. If $\tau(N)=N$, then the result concerning restriction follows immediately from the definition of shape. From now on, we assume $\tau(a)=N,\tau(N)=b$, where $a$ and $b$ may be distinct or equal. Further, let $a=p_{k}$, and assume the shape is of rank $M$. Finally we assume that we are considering a fixed choice of irreducible appearing in $\downarrow V$. We will see that our result does not depend on the choice of irreducible, so must hold for all irreducibles in the decomposition of $V$. To prove the general result, we will first show the following two properties:
\begin{enumerate}
	\item Rank $\downarrow V = M-2$, with $\downarrow P_{[M-2]} = P_{[M-1]}\setminus\{a\}$.
	\item $\downarrow \tau(i)=\tau(i)$, $\downarrow u(i)=u(i)$ for all $a\neq i<N$.
\end{enumerate}

	We start by noting that we must have
	\[
	Z_{\downarrow S,j} = Z_{S,j} \text{ for } j<k.\]
	To prove the first property, we first want to show that if $u(j)=0$ then we must have $\downarrow u(j)=0$, and so $\downarrow P\subseteq P_{[M]}\setminus N$. To see this, assume otherwise, and let 
	\[
	U=\{1\leq j< N\mid u(j)=0\text{ and }\downarrow u(j)\neq 0\}, \qquad Q\subset P_{[M]}\setminus N\text{ s.t. }Z_{Q\cup U,Q\cup U}=Z_{\downarrow S,\downarrow M}\neq 0,\]
	\[
	I=J=P_{[M]}\cup U, \qquad G=F\text{ s.t. }I_F=Q, \qquad K'=K\text{ s.t. }(I^F)_K=U.\]
	Then substituting into equation \eqref{EqMuirBr}, we get
	\[
	0 = \sum\limits_{P}(-q)^{\wt(P)-\wt(K)}\hat{R}^{Q\cup(I^F)_P,Q\cup(I^F)_P}_{P_{[M]},P_{[M]}}Z_{Q\cup U,Q\cup(I^F)_P}Z_{P_{[M]},Q\cup(I^F)^P},\]
	which by Lemma 3.16 of \cite{DCMo24b}, reduces to
	\[
	Z_{Q\cup U,Q\cup U}Z_{P_{[M]},P_{[M]}}=0,\]
	which is a contradiction. Therefore if $u(j)=0$, we must have $\downarrow u(j)=0$.
	
	Next, we want to show that $\downarrow S$ has rank $M-2$. The only possible rank $M-1$ minor in $\mcO_q(H(N-1))$ that can be non-zero is $Z_{P_{[M-1]},P_{[M-1]}}$. However we have
	\[
	Z_{P_{[M-1]},P_{[M-1]}} < Z_{\tau(P_{[M-1]}),P_{[M-1]}} = Z_{S,M-1},\]
	and so $Z_{P_{[M-1]},P_{[M-1]}}=0$. Therefore $\downarrow S$ can be at most rank $M-2$. To see that it is indeed rank $M-2$, we can apply equation \eqref{equation: weight relation 1} to $Z_{S,M-1}Z_{S,M-1}^{*}$ to get that 
	\[
	Z_{P_{[M-1]}\setminus \{a\}}\neq 0.\]
	To see that this is the smallest non-zero minor that will appear, we note that if there were some $x>a$ such that
	\[
	Z_{P_{[M-1]}\setminus\{x\}}\neq 0 \qquad \text{and} \qquad (P_{[M-1]}\setminus\{x\})<(P_{[M-1]}\setminus\{a\}),\]
	then it would require either $\downarrow\tau(a)=a$ or $\downarrow\tau(a)=y<N$. Either case would give
	\[
	(\downarrow P_{[k]},\downarrow\tau(P_{[k]}))<(P_{[k]},\tau(P_{[k]})),\]
	which is a contradiction, and therefore $Z_{P_{[M-1]}\setminus\{a\},P_{[M-1]}\setminus\{a\}}$ is the minimal rank $M-2$ minor in our chosen irreducible in $\downarrow V$.
	
	Hence we have shown $\downarrow S$ has rank $M-2$, and that $\downarrow u(a)=0$. We now need to show that all other $\downarrow\tau$ and $\downarrow u$ stay the same. If $\downarrow\tau(j)\neq\tau(j)$ for some $a\neq j<N$, then there would be a rank $k\leq m<M-2$ minor $Z_{Q,P_{[m+1]}\setminus \{a\}}$, with $Q\subset P_{[M-1]}\setminus \{a\}$, such that
	\[
	Z_{S,m} < Z_{Q,P_{[m+1]}\setminus \{a\}} < Z_{\tau(P_{[m+1]}\setminus \{a\}),P_{[m+1]}\setminus \{a\}}.\]
	To show that this does not occur, we use equation \eqref{EqMuirBr} with 
	\[
	I = P_{[m+1]}, \qquad J = Q\cup N, \qquad I_{F} = P_{[k-1]}, \qquad J_{G} = \tau(P_{[k-1]}), \qquad K = K^{\prime} = \{2,...,m-k+2\}.\]
	It can be seen that the only non-zero choice in the right hand side is given by $P=\{1,...,m-k+1\}$, and so the equation simplifies to give
	\[
	Z_{P_{[m+1]},Q\cup N}Z_{S,k-1}^{*} = (-q)^{k-m-1}Z_{P_{[m+1]}\setminus \{a\},Q}Z_{S,k}^{*}.\]
	If $Q<\tau(P_{[m+1]}\setminus \{a\})$, then $Q\cup N<\tau(P_{[m+1]})$, which would give zero on the left hand side. Hence no such $Q$ exists, and we must have $\downarrow\tau(j)=\tau(j)$ for all $a\neq j<N$. If we replace $Q$ with $\tau(P_{[m+1]}\setminus \{a\})$ in the above equation, we get
	\begin{equation}
		Z_{S,m+1}Z_{S,k-1} = (-q)^{k-m-1}Z_{\downarrow S,m}Z_{S,k}
	\end{equation}
	Then considering the unimodular part of each term, we see we must also have $\downarrow u(j)=u(j)$ for all $a\neq j<N$.
	
	Hence we have shown the two properties that we wanted, so now we can proceed to showing that $\tau$ must be an involution. Recall the definition of shape in \ref{def:shape}, and Theorem 3.6 of \cite{DCMo24b}, which states that $\tau$ must be a permutation. Now consider the case where $\tau$ is not an involution, so must contain a cycle of length $\geq 3$. Let $N$ be the largest element of this cycle, with $\tau(a)=N$, $\tau(N)=b$, and consider the restriction to $\mcO_q(H(N-1))$. From the first property above, we have $\downarrow u(a)=0$, $\downarrow\tau(a)=a$, whilst the second property gives $\downarrow\tau(j)=\tau(j)$ for all $a\neq j<N$. However, as $N\notin\downarrow \tau$ and $\downarrow\tau(a)=a$, we see that there is no element $j$ such that $\downarrow\tau(j)=b$. It follows that $\downarrow\tau$ is not a permutation if $\tau$ contains a cycle of length $\geq 3$, and therefore $\tau$ must be an involution.
	
	Finally, considering properties $1$ and $2$, along with $\tau$ being an involution, we see that we have also shown that the shape of our chosen irreducible in $\downarrow V$ is given by $\downarrow S=(S_{ij})_{1\leq i,j\leq N-1}$. As determination of the shape of the chosen irreducible in $\downarrow V$ did not depend on the choice of irreducible, it follows that each irreducible component of $\downarrow V$ must have the same shape, given by $\downarrow S$.
\end{proof}

\begin{Cor}\label{cor: self adjoint}
	The shape matrices are self-adjoint.
\end{Cor}	
\begin{proof}
	We have already established in Proposition \ref{proposition: restriction of shapes} that $\tau$ must be an involution, and so only need to show that $u_{\tau(i)}=\bar{u}_{i}$. For any $i$ with $\tau(i)=i$, it follows by restriction to an appropriate shape that $u_{i}$ must be real. Hence we just need to show that $u_{i}=\bar{u}_{\tau(i)}$ for $i\neq\tau(i)$. Consider equation \eqref{equation: weight relation 1}, and the corresponding unimodular part. we have:
	\[
	(-1)^{l(\tau_{\lvert P_{[k]}\cup\tau(i)})}(-1)^{l(\tau_{\lvert P_{[k]}\setminus\{i\}})}\prod\limits_{j\in P_{[k]}\cup\tau(i)}\bar{u}_{j}\prod\limits_{j\in P_{[k]}\setminus\{i\}}\bar{u}_{j} = -1.\]
	The length function has parity dependent on the number of generators in the permutation, and as the two permutations appearing differ by a single transposition, the two terms must combine to give $-1$. Further, by induction we can take all $u_{j}$ for $j\in P_{[k]}\setminus\{i\}$ to either be real or to cancel with $u_{\tau(j)}$. Hence we can simplify to give
	\[
	\bar{u}_{i}\bar{u}_{\tau(i)}=1.\]
\end{proof}
From now on, we will always assume that the shapes we are considering are self-adjoint.

For the highest weight theory, it will be useful to have a generalization of Lemma \ref{LemqCommpi}: 
\begin{Lem}
Given an $\mcO_q(H(N))$ shape $S$ of rank $M$, write
\[
Z_{\downarrow^{m}S,k} = Z_{\downarrow^{m}\tau(P_{[k]}),\downarrow^{m}P_{[k]}}.\]
Then in $\mcO_q(H(N))/I^{\leq}_S$, we have the following commutation relations:
\begin{equation}\label{equation: restricted ZSk comm rels}
Z_{\downarrow^{m}S,k}Z_{I,J} = q^{\lvert I\cap \downarrow^{m}P_{[k]}\rvert+\lvert I\cap\downarrow^{m}\tau(P_{[k]})\rvert-\lvert J\cap \downarrow^{m}P_{[k]}\rvert-\lvert J\cap\downarrow^{m}\tau(P_{[k]})\rvert}Z_{I,J}Z_{\downarrow^{m}S,k}.
\end{equation}
\end{Lem}
\begin{proof}
Assume we have some $Z_{\downarrow^{m}S,k}\neq Z_{S,k}$, then considering the left hand side of equation \eqref{EqGenCommRel} with $I=\downarrow^{m}P_{[k]}$, $J=\downarrow^{m}\tau(P_{[k]})$, we have
\[
\sum\limits_{K,L,L^{\prime}}\left(\sum\limits_{P^{\prime}}\hat{R}^{P^{\prime}I^{\prime}}_{\downarrow^{m}\tau(P_{[k]}),K}\hat{R}^{\downarrow^{m}P_{[k]}, L}_{P^{\prime}L^{\prime}}\right)Z_{KL}Z_{L^{\prime},J^{\prime}}\]
The sum is over $L$ and $K$ such that 
\[
L\preceq\left(\downarrow^{m}P_{[k]}\right), \qquad K\preceq\left(\downarrow^{m}\tau(P_{[k]})\right),\]
where $\preceq$ denotes the partial order defined in \eqref{EqOrdPart}. In particular, as 
\[
N,...,N-m+1\notin \left(\downarrow^{m}P_{[k]}\right)\cup\left(\downarrow^{m}\tau(P_{[k]})\right),\]
we only need to consider $K,L\subseteq[N-m]$. If then follows from this that the only $K,L$ that give a non-zero $Z_{KL}$ is 
\[
K = \downarrow^{m}\tau(P_{[k]}), \qquad L = \downarrow^{m}P_{[k]}.\]
We find a similar requirement for the right hand side, and so equation \eqref{EqGenCommRel} reduces to the required result.
\end{proof}

We now introduce our definition of highest weight vector:

\begin{Def}
Given an $\mcO_q(H(N))$-representation $V$, we call $v\in V$ a weight vector if it is an eigenvector for each $Z_{\downarrow^{m}S,k}$, $0\leq m\leq N-2$. To each weight vector $v$, we assign to it an ordered weight $w(v)\in\mathbb{R}^{M+\downarrow M+...}$ consisting of the non-unimodular parts of the eigenvalues of the $Z_{\downarrow^{m}S,k}$, ordered as follows:
\[
(Z_{\downarrow^{N-2}S,1}Z_{\downarrow^{N-2}S,2},Z_{\downarrow^{N-3}S,1},...,Z_{S,M}).\]
We call $v_{0}$ a highest weight vector if $w(v_{0})$ is maximal lexicographically in the set of all $w(v)$. 
\end{Def}

By our choice of definition, it immediately follows that a highest weight vector of an $\mcO_{q}(H(N))$ representation $V$ is also a highest weight vector in $\downarrow V$. From boundedness of the representation, it is also immediate that a highest weight vector necessarily exists. We want to show that when $V$ is irreducible, a highest weight vector is unique up to a scalar.

Before proving uniqueness, we need to describe in more detail how different elements will act in the representation. To simplify things, in what follows we will consider only rank one $Z_{ij}$. By Lemma \ref{LemqCommpi}, we see that any self-adjoint element will commute with the $Z_{S,k}$. Further from the $N=2$ classification in Section 2.2 of \cite{DCMo24a}, we see that $Z_{11}$ and $Z_{22}$ act as scalars on weight vectors. As $\sigma_{1}$ will also act as a scalar on irreducibles, by induction any $Z_{ii}$ can be taken to act as scalars on weight vectors. For $Z_{ij}$, $i\neq j$, if it $q$-commutes with at least one $Z_{S,k}$, it can be considered as a lowering or raising operator. Note that for general shapes, the $q$-powers appearing may be positive for one $Z_{S,k}$ and negative for another $Z_{S,k^{\prime}}$, and so whether the given $Z_{ij}$ is viewed as a lowering or raising operator becomes more complicated.

For general shapes, there may also be $Z_{ij}$, $i\neq j$, that commute with all $Z_{S,k}$. We will need to give a method for dealing with these to show the uniqueness of highest weight vectors. We start by describing explicitly which elements of $\mcO_q(H(N))$ will commute with all $Z_{S,k}$:
\begin{Prop}
	Let $Q\subset P_{[M]}$, and $S,T\subseteq u_{0}$, $\lvert S\rvert=\lvert T\rvert$, where $u_{0}:=\{1\leq i\leq N:u(i)=0\}$. Then every non-self adjoint element $Z_{IJ}\notin \{Z_{S,k}\}$ that commutes with all $Z_{S,k}$ is of the form
	\[ Z_{Q\cup S,\tau(Q)\cup T}\]
	where we allow $S=T$ only if $\tau(Q)\ne Q$, and we further allow $S=T=\emptyset$ only if $\tau(Q)\neq Q$ and $Q,\tau(Q)\neq P_{[k]}$.
\end{Prop}
\begin{proof}
This follows from the commutation relations given by equation \eqref{EqqCommQuot}.
\end{proof}

To simplify notation, we introduce the following sets:

\begin{Def}
For a given shape, we denote by $Z_{0}^{S}$ the set of $Z_{IJ}$, $I\neq J$, such that $Z_{IJ}$ commutes with all $Z_{S,k}$. We further denote by $Z_{+}^{S}$ the set of $Z_{IJ}$, $I\neq J$ such that
\[
Z_{S,k}Z_{IJ} = q^{n_{k}}Z_{IJ}Z_{S,k},\]
for $n_{k}\in\mathbb{N}$, with at least one $n_{k}>0$. $Z_{-}^{S}$ is defined similarly for $n_k\in -\mathbb{N}$, with at least one $n_{k}<0$, or alternatively we can take $Z^{S}_{-}=(Z^{S}_{+})^{*}$. Finally we denote by  $Z_{d}^{S}$ the set
\[
\{Z_{S,k},Z_{S,k}^{*},Z_{S,k}^{-1},(Z_{S,k}^{-1})^{*}: 1\leq k\leq M\}.\]
\end{Def}

We want to show that we can rewrite elements of $Z_{0}^{S}$ in terms of elements that $q$-commute. 

\begin{Lem}\label{lemma: ZS0 expansion}
	Let $Z_{ij}$ be a rank 1 element of $Z^{S}_{0}$, then it can be written in terms of $Z^{S}_{\pm,d}$, in the form
	\[
	\sum Z^{S}_{+}Z^{S}_{-}Z^{S}_{d}.\]
\end{Lem}
\begin{proof}
	We can assume $i>j$, and either $i=\tau(j)$ or $i,j\in u_{0}$. In either case there is a minimal $1\leq m\leq M$ such that $Z_{i\cup \tau(P_{[m]}),j\cup P_{[m]}}\in\{Z_{S,k}\}$ or $\in \mathcal{I}_{S}$. Let $I:=i\cup\tau(P_{[m]})$, $J:=j\cup P_{[m]}$, choose $K$ such that $J_{K}=j$, and consider this substituted into equation \eqref{EqLaplExp3}. 
	
	$A$ is a single element set, so we must have either $A=I_{P}$ and $B=I^{P}$, or else $A\in I^{P}$ and $B=(I^{P}\setminus A)\cup I_{P}$. Hence we can write $A=I_{Q}$ and $B=I^{Q}$. Further, as $j\notin I$, we then must have $C=j$, $D=I^{Q}$, and so equation \eqref{EqLaplExp3} becomes:
	\[
	Z_{IJ}=\sum_{P,Q,Q\geq P}(-q)^{wt(P)-wt(K)}(\hat{R}^{-1})^{I_{Q},I_{P}}_{I^{P},I^{Q}}\hat{R}^{j,j}_{I^{Q},I^{Q}}Z_{I_{Q},j}Z_{I^{Q},P_{[m]}}\]
	There is some $Q_{i}$ such that $I_{Q_{i}}=i$, then for any $Q>Q_{i}$, $Z_{I^{Q},P_{[m]}}=0$. Hence we can write:
	\begin{align}
		\sum\limits_{P}(-q)^{P-K}(\hat{R}^{-1})^{i,I_{P}}_{I^{P},\tau(P_{[m]})}Z_{ij}&=\left(Z_{IJ}-\sum\limits_{P,Q\geq P}^{Q_{i}-1}(-q)^{P-K}(\hat{R}^{-1})^{I_{Q},I_{P}}_{I^{P},I^{Q}}Z_{I_{Q},j}Z_{I^{Q},P_{[m]}}\right)Z_{S,m}^{-1}
	\end{align}
	The sum on the left hand side is the same identity that appeared in the left hand side of equation (3.7) of \cite{DCMo24b}, and which was shown to be non-zero in Lemma 3.18 of \cite{DCMo24b}. On the right hand side, for any $Q\neq Q_{i}$, we must have $Z_{I^{Q},P_{[m]}}\in Z^{S}_{-}$, and $Z_{I_{Q},j}\in Z^{S}_{+}$. Hence we have written $Z_{ij}$ in the required form.
\end{proof}
\begin{Theorem}\label{theorem: unique highest weight}
	Every bounded irreducible $\mcO_q(H(N))$ $*$-representation on a pre-Hilbert space has a unique highest weight vector.
\end{Theorem}
\begin{proof}
	Assume there is an irreducible representation $V$ with highest weight vectors $v$, $v^{\prime}$. Then there must be some $z\in\mcO_q(H(N))/I^{\leq}_S$ such that $zv=v^{\prime}$. We first consider $z$ to be a rank one generator. As $v$ is a highest weight vector, if $z\in Z^{S}_{-}$, then $zv=0$. Similarly, if $z\in Z^{S}_{+}$, then $z^{*}v^{\prime}=0$. More generally, if $z$ $q$-commutes non-trivially with $Z_{S,k}$ then it must act as a lowering or raising operator, and cannot map between $z$ and $z^{\prime}$. If $z\in Z^{S}_{0}$, then by Lemma \ref{lemma: ZS0 expansion}, either $zv=0$, or $zv=\lambda v$ for some scalar $\lambda$. The only remaining options are $z=Z_{S,k}$ for some $k$, or $z=Z_{ii}$ for some $i$, both of which we know will act as scalars on weight vectors. 
	
	Now consider the case when $z$ is a higher rank element. We can take $z$ to be a product of rank 1 generators of the form 
	\[
	Z_{i_{1},j_{1}}Z_{i_{2},j_{2}}...Z_{i_{m},j_{m}}.\]
	From the $N=2$ classification in Section 2.2 of \cite{DCMo24a}, we know that every $\mcO_{q}(H(2))$ irreducible has a unique highest weight vector. We can then proceed by induction. Assume true for $N-1$, then if an $\mcO_{q}(H(N))$ irreducible contains multiple highest weight vectors, they must lie in different $\mcO_{q}(H(N-1))$ irreducibles. It follows for $z$ to map between multiple $\mcO_{q}(H(N))$ highest weight vectors, it must contain $Z_{lN}$ or $Z_{Nl}$ for some $1\leq l<N$. However, from the $\mcO_{q}(H(N))$ relations, we know we can always pull $Z_{lN}$ or $Z_{Nl}$ to the left or rightmost position, so that it acts directly on $v$ or $v^{\prime}$. This then reduces to the rank one case. It then follows that $V$ has a unique highest weight vector. 
\end{proof}

\section{The $\mathcal{O}_{q}(U(N))$ coaction on shapes}\label{section: coaction on shapes}

Recall the coaction $\alpha_{i}:\mcO_q(H(N))\rightarrow\mcO_q(H(N))\otimes B(l^2(\N))$, defined by equation \eqref{def: coaction UN}. If we consider $\alpha_{i}$ applied to an irreducible $\mcO_q(H(N))$ representation $V$, then in certain cases $\alpha_{i}$ can change the shape of the resulting representation. We want to describe this in more detail. 

We denote $I_{(k)}:=I\cap\{k,k+1\}$, we also denote by $I^{(k)}$ the set given by permuting $k$ and $k+1$ in $I$. Using the general coaction formula \eqref{equation: qminor coaction}, it is straightforward to see that in terms of the quantum minors, the coaction $\alpha_{k}$ is given by:
\begin{align}\label{equation: SU2 coaction}
	\alpha_{k}(Z_{IJ})=& Z_{IJ}\otimes X^{*}_{I_{(k)}I_{(k)}}X_{J_{(k)}J_{(k)}}+Z_{I^{(k)}J}\otimes X^{*}_{I^{(k)}_{(k)}I_{(k)}}X_{J_{(k)}J_{(k)}} \nonumber\\
	&+Z_{IJ^{(k)}}\otimes X^{*}_{I_{(k)}I_{(k)}}X_{J^{(k)}_{(k)}J_{(k)}}+Z_{I^{(k)}J^{(k)}}\otimes X^{*}_{I^{(k)}_{(k)}I_{(k)}}X_{J^{(k)}_{(k)}J_{(k)}}
\end{align}

where we take $X_{I_{(k)}I_{(k)}}=1$, $X_{I_{(k)}^{(k)}I_{(k)}}=0$ if $I_{(k)}=\emptyset$. From this we see that for $\alpha_{k}$ to change the shape of a representation, there must be a $Z_{IJ}=0$ and a $Z_{I^{\prime}J^{\prime}}\neq 0$, $(J,I)<(J^{\prime},I^{\prime})$, such that $(I,J)$ is obtained from $(I^{\prime},J^{\prime})$ by permuting $k$ and $k+1$.

We can now describe how $\alpha_{k}$ changes the shape in terms of the shape matrix.
\begin{Theorem}\label{theorem: Uq coaction}
	Let $V$ be an irreducible representation, $S$ its shape matrix, and let $c_{k}$ be the permutation matrix corresponding to the transposition $(k,k+1)$. If $\tau(k)=k+1$, then $\alpha_{k}V$ will contain two irreducibles, whose shapes are given by replacing the $(k,k+1)$ sub-matrix of $S$ by the matrix $\text{Diag}\{\pm 1,\mp 1\}$. Otherwise, if the shape of $\alpha_{k}V$ is different to the shape of $V$, then it is given by $c_{k}Sc_{k}$. 
\end{Theorem}
\begin{proof}
	
	We first note that when $\tau(i)\neq i$, we can freely choose the unimodular part of the $\mcO_q(SU(2))$ irreducible to give a required $u(i)$ under the coaction. Hence we have made the choice to have the coaction preserve unimodular parts when $\tau(i)\neq i$.
	
	For $N=2$, there are two choices of shape to consider, the first is $u(1)=0$, $u(2)\neq 0$, and the second is $\tau(1)=2$. Either case will result in $Z_{11}=0$, however we have
	\[
	\alpha_{1}(Z_{11}) = Z_{11}\otimes X_{11}^{*}X_{11}+Z_{12}\otimes X_{11}^{*}X_{21}+Z_{21}\otimes X_{21}^{*}X_{11}+Z_{22}\otimes X^{*}_{21}X_{21}.\]
	Hence the resulting representation has $Z_{11}\neq 0$, and so its shape must be as described by the theorem.
	
	We now proceed by induction on $N$: For $\alpha_{k}$ with $1\leq k\leq N-2$, consider the restriction to the $\mcO_q(H(N-1))$ shape as described in Proposition \ref{proposition: restriction of shapes}. Restriction and coaction on shapes will commute, so we only need to consider the case when $\alpha_{k}$ affects $\tau(N)$. It is straightforward to see that this will only occur if $\tau(N)=k$ and $u(k+1)\neq 0$, in which case we will have $\downarrow u(k)=0$, $\alpha_{k}\left(\downarrow u(k)\right)\neq 0$. It then follows that we must have $\alpha_{k}\tau(N)=k+1$, and so the coaction has changed the shape as described.
	
	Next, consider $\alpha_{N-1}V$. We will consider three different cases: $\tau(N)=N$, $\tau(N)=i<N-1$, and $\tau(N)=N-1$. In the first case, for the shape to change we must have $u(N-1)=0$, so that $Z_{S,m}=Z_{I\cup\{N\},J\cup\{N\}}$, and we have
	\[
	\alpha_{N-1}(Z_{S,m}) = \alpha_{N-1}(Z_{I\cup\{N-1\},J\cup\{N-1\}}) = Z_{I\cup\{N\},J\cup\{N\}}\otimes X_{N,N-1}^{*}X_{N,N-1}\]
	Hence the new shape has the values of $\tau,u$ on $N-1,N$ switched, which can be seen to come from $c_{N-1}Sc_{N-1}$.
	
	For the second case, let $p_{m}=i$, and consider $Z_{S,m}$. Then as $\tau(N)=i<N-1$, we have $N-1,N\notin P_{[m]}$. If $\tau(N-1)=j<i$, with $p_{m^{\prime}}=j$, then, noting that $\alpha_{N-1}$ must affect $\tau,u,$ of $N-1$ and $N$, we see $Z_{S,m^{\prime}}$ will not change and so the shape stays the same. However, if $\tau(N-1)=j>i$, then
	\[
	\alpha_{N-1}(Z_{S,m}) = \alpha_{N-1}(Z_{\tau(P_{[m]})\cup\{N-1\}\setminus N,P_{[m]}}) = Z_{\tau(P_{[m]}),P_{[m]}}\otimes X_{21}^{*}.\]
	Hence we now have $\tau(N-1)=i$, and the shape matrix can be seen to transform in the required way.
	
	Finally, for the third case, if $p_{m}=N-1$, we see that 
	\[
	\alpha_{N-1}(Z_{\tau(P_{[m]})\cup\{N-1\}\setminus N,P_{[m]}})\neq 0,\]
	so we must have $\tau(N)=N-1$ replaced with $\tau(N-1)=N-1$, $\tau(N)=N$. We will show that both sign choices for $u(N-1,N)$ appear in Section \ref{section: Diag}. 
\end{proof}
Apart from the special case where $\tau(k)=k+1$, every other case of shape change will involve the tensor product of diagonal operators. Hence it can be seen that if $V$ is an irreducible of shape $S$, then every irreducible in the decomposition of $\alpha_{k}S$ will have the same shape as described in Theorem \ref{theorem: Uq coaction}. We will show later that this is also true for the $\tau(k)=k+1$ special case.
\begin{Cor}\label{coaction to big cell}
	Let $V$ be an irreducible $\mcO_q(H(N))$ representation. Then we can choose a sequence $\{i_{1},...,i_{m}\}$ such that $\alpha_{i_{m}}...\alpha_{i_{1}}V$ is a big cell representation.  
\end{Cor}
\begin{Def}\label{def: shape signature}
If $S$ is a shape and there is an $\mcO_q(U(N))$ coaction that takes $S$ to a big cell representation of shape $\epsilon$ and signature $\zeta_\epsilon$, then we say $S$ has \emph{signature} $\zeta_\epsilon$.
\end{Def}
By a combination of weak containment and coinvariance of the centre, this definition of signature for shapes can be seen to be well-defined.

\subsection{Diagonalizing the $\alpha_{k}$ coaction when $\tau(k)=k+1$}\label{section: Diag}

We want to show that when we act $\alpha_{k}$ on a representation with shape such that $\tau(k)=k+1$, the result will contain two irreducible representations with the $k,k+1$ sub matrix of their shapes corresponding to $\text{Diag}\{\pm 1,\mp 1\}$, with both sign choices appearing. 

To begin, we will need the following lemma:

\begin{Lem}\label{lemma: gen relation}
	If $\tau(k)=k+1$, with $k=p_{m}$, and $\tau(P_{[m-1]})=P_{[m-1]}$, then we have:
	\begin{align}
		&Z_{P_{[m]},\tau(P_{[m]})}Z_{\tau(P_{[m]}),P_{[m]}} \nonumber\\ =&-q^{2}Z_{P_{[m+1]},P_{[m+1]}}Z_{P_{[m-1]}P_{[m-1]}}+(q^{2}Z_{P_{[m]},P_{[m]}}+Z_{\tau(P_{[m]}),\tau(P_{[m]})})Z_{P_{[m]},P_{[m]}}-Z_{P_{[m]},P_{[m]}}^{2}\\
		&Z_{\tau(P_{[m]}),P_{[m]}}Z_{P_{[m]},\tau(P_{[m]})} \nonumber\\ =&-q^{2}Z_{P_{[m+1]},P_{[m+1]}}Z_{P_{[m-1]}P_{[m-1]}}+q^{2}(q^{2}Z_{P_{[m]},P_{[m]}}+Z_{\tau(P_{[m]}),\tau(P_{[m]})})Z_{P_{[m]},P_{[m]}}-q^{4}Z_{P_{[m]},P_{[m]}}^{2}
	\end{align}
\end{Lem}
\begin{proof}
	Consider equation \eqref{EqMuirBr} with $I=J=P_{[m+1]}$, $K=K^{\prime}=1$, and $F=G=[m-1]$ such that $I_{F}=P_{[m-1]}$, $I^{F}=\{k,k+1\}$, $(I^{F})_{K}=k$. Then, denoting $\tau:=\tau(P_{[m]})$ for short, we have:
	\begin{align*} 
		Z_{P_{[m+1]},P_{[m+1]}}Z_{P_{[m-1]},P_{[m-1]}}
=&q^{-m}(\widehat{R}^{-1})^{\tau,P_{[m]}}_{\tau,P_{[m]}} Z_{\tau,P_{[m]}}Z_{P_{[m]},\tau}+\left(1-q^{2-m}(\widehat{R}^{-1})^{\tau,P_{[m]}}_{\tau,P_{[m]}}\right)  Z_{\tau,\tau} Z_{P_{[m]},P_{[m]}}\\
&-q^{m} \widehat{R}^{\tau,P_{[m]}}_{\tau,P_{[m]}} Z_{P_{[m]},P_{[m]}}^2- Z_{P_{[m]},\tau} Z_{\tau,P_{[m]}}
	\end{align*}
	Next, using equation \eqref{EqGenCommRel} with $J=J^{\prime}=\tau(P_{[m]})$, $I=I^{\prime}=P_{[m]}$, we get:
	\begin{align*}
		Z_{\tau,P_{[m]}}Z_{P_{[m]},\tau}=&q^{m}\widehat{R}^{\tau,P_{[m]}}_{\tau,P_{[m]}}Z_{P_{[m]},P_{[m]}}^{2}+Z_{P_{[m]},\tau}Z_{\tau,P_{[m]}}-q^{m}\widehat{R}^{\tau,P_{[m]}}_{\tau,P_{[m]}}Z_{P_{[m]},P_{[m]}}Z_{\tau\tau}
	\end{align*}
	By using equation (2.9) of \cite{DCMo24b}, with  $I=J=P_{[m]}$, $I^{\prime}=J^{\prime}=\tau(P_{[m]})$, (or $I=J=\tau(P_{[m]})$, $I^{\prime}=J^{\prime}=P_{[m]}$ for the inverse), then quotienting $\mcO_q(U(N))\rightarrow\mcO_q(SU(2))$, $X_{kk}\mapsto X_{11}$, we find that 
	\[
	\widehat{R}^{\tau,P_{[m]}}_{\tau,P_{[m]}}=q^{-m}-q^{2-m}, ~ (\widehat{R}^{-1})^{\tau,P_{[m]}}_{\tau,P_{[m]}}=q^{m}-q^{m-2}.\]	
	Combining everything together and simplifying we get the result.
\end{proof}
We can use the previous lemma to find the highest weight of non-diagonal cases as follows:
\begin{Prop}\label{prop: highest weight diag}
	Let $V$ be an irreducible representation with shape such that $\tau(k)=k+1$, with $p_{m}=k$, and $\tau(P_{[m-1]})=P_{[m-1]}$. Let $v_{0}$ be the highest weight vector of $V$, and $w_{m-1},w_{m+1},w_{\tau}$ denote the values of $Z_{P_{[m-1]},P_{[m-1]}}$, $Z_{P_{[m+1]},P_{[m+1]}}$, $Z_{\tau(P_{[m]}),\tau(P_{[m]})}$ on $v_{0}$ respectively. Then the possible highest weights of $\alpha_{k}(Z_{P_{[m]},P_{[m]}})$ on $\alpha_{k}V$ are given as follows:
	\begin{align}
		\frac{1}{2}\left(w_{\tau}\pm\sqrt{w_{\tau}^{2}-4q^{2}w_{m+1}w_{m-1}}\right)
	\end{align}
\end{Prop}
\begin{proof}
	The highest weight vector for $\alpha_{k}V$ must be of the form $v_{0}\otimes L_{0}$, where $L_{0}$ is an element of the $\mathcal{O}_{q}(SU(2))$ representation. Let
	\[
	\alpha_{k}(Z_{P_{[m]},P_{[m]}})(v_{0}\otimes L_{0}) = xv_{0}\otimes L_{0}\]
	Note that  
	\[
	\alpha_{k}(Z_{P_{[m+1]},P_{[m+1]}})=Z_{P_{[m+1]},P_{[m+1]}}\otimes 1, ~ \alpha_{k}(Z_{P_{[m-1]},P_{[m-1]}})=Z_{P_{[m-1]},P_{[m-1]}}\otimes 1\]
	\[\alpha_{k}(q^{2}Z_{P_{[m]},P_{[m]}}+Z_{\tau(P_{[m]}),\tau(P_{[m]})})=(q^{2}Z_{P_{[m]},P_{[m]}}+Z_{\tau(P_{[m]}),\tau(P_{[m]})})\otimes 1, \]
	as well as that $\alpha_{k}(Z_{\tau(P_{[m]}),P_{[m]}})$ will act as zero on a highest weight vector. Then applying $\alpha_{k}$ to Lemma \ref{lemma: gen relation}, we get:
	\begin{align*}
		0=&q^{2}\left(Z_{P_{[m+1]},P_{[m+1]}}Z_{P_{[m-1]},P_{[m-1]}}v_{0}\right)\otimes L_{0}-\alpha_{k}(Z_{P_{[m]},P_{[m]}})\left(Z_{\tau(P_{[m]}),\tau(P_{[m]})}v_{0}\right)\otimes L_{0}\\
		&+\alpha_{k}(Z_{P_{[m]},P_{[m]}}^{2})v_{0}\otimes L_{0}
	\end{align*}
	Hence we have:
	\begin{align*}
		q^{2}w_{m+1}w_{m-1}-w_{\tau}x+x^{2}&=0
	\end{align*}
	Solving for $x$ we get the possible highest weights of $\alpha_{k}(Z_{P_{[m]},P_{[m]}})$.
\end{proof}
We now need to show that both highest weight solutions appear in $\alpha_{k}V$. 
\begin{Prop}\label{prop: diagonalize tau(k)}
	If an irreducible $V$ has a shape with $\tau(k)=k+1$, $p_{m}=k$, and $\tau(P_{[m-1]})=P_{[m-1]}$, then $\alpha_{k}V$ contains two highest weight vectors corresponding to the two possible highest weights for $\alpha_{k}(Z_{P_{[m]},P_{[m]}})$.
\end{Prop}
\begin{proof}
	Recall the $\mathcal{O}_{q}(SU(2))$ representation $\mbs$. For simplicity, we denote its action by
	\[
	X_{11}e_{n}=a_{n}e_{n-1}, ~ X_{12}e_{n}=b_{n}e_{n}, ~ X_{21}e_{n}=c_{n}e_{n}, ~ X_{22}e_{n}=d_{n}e_{n+1}.\] 
	Let $v_{0}$ denote the highest weight vector of $V$, and put 
	\[
	Z_{\tau(P_{[m]}),P_{[m]}}v_{0}=fv_{0}, \qquad Z_{\tau(P_{[m]}),\tau(P_{[m]})}v_{0}=w_{\tau}v_{0}, \qquad Z_{P_{[m+1]},P_{[m+1]}}Z_{P_{[m-1]},P_{[m-1]}}v_{0}=w_{m+1}w_{m-1}v_{0}. \] 
	The highest weight vector for $\alpha_{k}V$ must be of the form $v_{0}\otimes L_{0}$, where
	\[
	L_{0}=\sum\limits_{i=0}\lambda_{i}e_{i}\]
	From Proposition \ref{prop: highest weight diag}, we know that
	\[
	\alpha_{k}(Z_{P_{[m]},P_{[m]}})(v_{0}\otimes L_{0})=xv_{0}\otimes L_{0}, \qquad x = \frac{1}{2}\left(w_{\tau}\pm\sqrt{w_{\tau}^{2}-4q^{2}w_{m+1}w_{m-1}}\right)\] 
	However, we also have
	\[
	\alpha_{k}(Z_{P_{[m]},P_{[m]}}) = Z_{\tau(P_{[m]}),P_{[m]}}\otimes X_{21}^{*}X_{11}+Z_{P_{[m]},\tau(P_{[m]})}\otimes X^{*}_{11}X_{21}+Z_{\tau(P_{[m]}),\tau(P_{[m]})}\otimes X^{*}_{21}X_{21}\]
	which gives
	\[
	\alpha_{k}(Z_{P_{[m]},P_{[m]}})(v_{0}\otimes L_{0})=\sum\limits_{i}\left(\bar{f}c_{i-1}d_{i-1}\lambda_{i-1}-q^{-1}fb_{i}a_{i+1}\lambda_{i+1}-q^{-1}w_{\tau}b_{i}c_{i}\lambda_{i}\right)v_{0}\otimes e_{i}.\]
	Hence we require
	\[
	\lambda_{i+2} = \frac{q\bar{f}c_{i}d_{i}\lambda_{i}-w_{\tau}b_{i+1}c_{i+1}\lambda_{i+1}-qx\lambda_{i+1}}{fb_{i+1}a_{i+2}}.\]
	For $v_{0}\otimes L_{0}$ to be a highest weight vector, we will also need $\alpha_{k}(Z_{\tau(P_{[m]}),P_{[m]}})$ to act as zero. We have
	\[
	\alpha_{k}(Z_{\tau(P_{[m]}),P_{[m]}})v_{0}\otimes L_{0} = \sum\limits_{i=0}\left(fa_{i+2}a_{i+1}\lambda_{i+2}-q\bar{f}c_{i}^{2}\lambda_{i}+w_{\tau}c_{i+1}a_{i+1}\lambda_{i+1}\right)v_{0}\otimes e_{i}.\]
	Then we also require
	\[
	\lambda_{i+2}= \frac{q\bar{f}c_{i}^{2}\lambda_{i}-w_{\tau}c_{i+1}a_{i+1}\lambda_{i+1}}{fa_{i+1}a_{i+2}}.\]
	Combining this with the previous formula for $\lambda_{i+2}$, we get
	\[
	\lambda_{i+1} = \frac{\bar{f}c_{i}\lambda_{i}}{xa_{i+1}}.
	\]
	It follows that the two solutions for $x$ will give two linearly independent highest weight vectors $v_{0}\otimes L_{0}^{\pm}$, which both appear in $\alpha_{k}V$. By Theorem \ref{theorem: unique highest weight}, we know each irreducible has a unique highest weight vector, and so an irreducible corresponding to each solution appears in $\alpha_{k}V$.
\end{proof}

\begin{Rem}
	We note that one solution for $\alpha_{k}(Z_{P_{[m]},P_{[m]}})$ will be positive, and the other negative, so in terms of shapes, the two solutions will have the $(k,k+1)$ submatrix differ by a sign, as described in Theorem \ref{theorem: Uq coaction}.
\end{Rem}

Using similar techniques, it can also be seen that no other shapes appear via the coaction:
\begin{Prop}
If an irreducible $\mcO_q(H(N))$ representation $V$ has shape $S$ such that $\tau(k)=k+1$ with $p_m=k$, and $\tau(P_{[m-1]})=P_{[m-1]}$, then the only shapes that appear in $\alpha_{k}S$ are those as described by Theorem \ref{theorem: Uq coaction}.
\end{Prop}
\begin{proof}
The only other shape that could potentially appear in $\alpha_{k}S$ is $S$ itself. Assume some irreducible $W$ of shape $S$ did appear in the decomposition of $\alpha_{k}S$. Then we can assume there is some $v\in V$ such that $W$ is generated by 
\[
L:=\sum\limits_{i=0}^{\infty} v\otimes \mu_{i}e_{i}, \qquad \mu_{i}\in\C.\]
As $W$ is of shape $S$, we must have
\[
\alpha_{k}(Z_{P_[m],\tau(P_{[m]})})L=yL, \qquad y\in\C.\]
Using the same notation as the proof of Proposition \ref{prop: diagonalize tau(k)}, then expanding and solving for $\mu_i$, we get
\[
\sum\limits_{i=0}^{\infty}cv\otimes\mu_{i}e_{i} = \sum\limits_{i=0}^{\infty}v\otimes\left(\mu_{i-2}\bar{f}d_{i-2}d_{i-1}-q^{-1}\mu_{i}fb_{i}^{2}-q^{-1}\mu_{i-1}w_{\tau}b_{i}d_{i-1}\right)e_{i}.\]
In particular, via the $i=0$ term, we get
\[
y = -q^{-1}fb_{0}^2 = -qf.\]
On the other hand, 
\[
\alpha_{k}\left(Z_{P_{[m-1]},P_{[m-1]}}Z_{P_{[m+1]},P_{[m+1]}}\right)=Z_{P_{[m-1]},P_{[m-1]}}Z_{P_{[m+1]},P_{[m+1]}}\otimes 1,\]
so by equation \eqref{equation: weight relation 1}, for $W$ to appear would require
\[
\alpha_{k}(Z_{P_{[m]},\tau(P_{[m]})})L=fL.\]
Hence we have a contradiction and therefore $W$ can not exist in $\alpha_{k}S$.
\end{proof}

\section{Highest weights for a given shape}\label{section: highest weights for shapes}

Given an irreducible representation of a chosen shape, we can use the $\mcO_q(U(N))$ coaction to construct a big cell representation from the irreducible. This then leaves the question of whether we can relate the highest weight of the big cell representation to the highest weight of the starting irreducible.

Before proceeding further, we will need to introduce the following definitions:
\begin{Def}
	Given a shape of rank $M$, let $t_{i}$ denote a cycle in the decomposition of $\tau(P_{[M]})$. We define an ordering on the set of $t_{i}$ by $t_{i}\vartriangleleft t_{j}$ if $Max(t_{i})<Max(t_{j})$ and label the $t_{i}$ accordingly. Given a chosen weight $r\in\mathbb{R}^{M}$, we define a map
	\[ W_{r}:\{t_{i}\}\rightarrow \mathbb{R}^{M}\]
	as follows. Partition $r$ into subsets $\mathfrak{r}_{i}$ with respect to the standard order on $\mathbb{R}^{M}$ such that $\lvert \mathfrak{r}_{i}\rvert=\lvert t_{i}\rvert$. Then we define
	\[
	W_{r}(t_{i})=\sum\limits_{r_{j}\in\mathfrak{r}_{i}} r_{j}.\]	
\end{Def}	
\begin{Def}
	Given a fixed shape of rank $M$, for $I,J\subseteq P_{[M]}$, we define the \textit{closure} of $Z_{IJ}$, denoted $Z_{IJ}^{C}$, to be the element $Z_{KK}$, where	\[K=\bigcup\limits_{t_{i}\trianglelefteq t_{j}} t_{i},\]
	and $t_{j}$ is the maximal choice such that
	\[t_{j}\cap\left(I\cup J\right)\neq\emptyset.\]	
\end{Def}	
\begin{Def}
	Given a shape, take the subset $\mathfrak{T}_{\tau}:=\{t_{i}:\lvert t_{i}\rvert =2\}$. We denote by $W_{\epsilon}$ the set 
	\[
	\{j:W_{r}(t_{i})=r_{j-1}+r_{j}, t_{i}\in \mathfrak{T}_{\tau}\}.\]	
\end{Def}
Finally we will also need the following:
\begin{Def}
	Given $Z_{S,k}$ for a fixed shape, let $\mathfrak{N}_{S,k}$ be the set
	\[
	\{t_{i}:j\in t_{i} \text{ with } j\in P_{[k]},\tau(j)\notin P_{[k]}\}.\]	
	We denote by $\mathfrak{N}_{S,k}^{i}$ the $i$ largest elements of $\mathfrak{N}_{S,k}$. We define the set $\mathfrak{C}_{S,k}=\{\mathfrak{c}_{i}\in\mathbb{N}, i<\lvert\mathfrak{N}_{S,k}\rvert\}$ as follows:
	
	Let $\mathfrak{c}_{0}$ be the rank of $Z_{S,k}^{C}$, and let $\mathfrak{c}_{i}$ be the rank of
	\[\left(Z_{\tau(P_{[k]})\setminus\mathfrak{N}_{S,k}^{i},P_{[k]}\setminus\mathfrak{N}_{S,k}^{i}}\right)^{C}.\]
\end{Def}
\begin{Exa}
	Consider the shape
	\[\tau=(4,5,3,1,2), \qquad u(i)\neq 0\text{ for all }i,\]
	which under the ordering of $\tau-$cycles has
	\[
	\{3\}\vartriangleleft\{1,4\}\vartriangleleft\{2,5\}\]
	The function $W_{r}$, and the set $W_{\epsilon}$ are given by
	\[ W_{r}(3)=r_{1}, \quad W_{r}(\{1,4\})=r_{2}+r_{3}, \quad  W_{r}(\{2,5\})=r_{4}+r_{5}, \quad W_{\epsilon}=\{3,5\}.\]
	The closure of $Z_{S,1}=Z_{41}$ is $Z_{134,134}$, and the closure of $Z_{S,k}$ for $k=2,3$ or $4$ is $Z_{[5]}$. Taking $Z_{S,3}=Z_{345,123}$, then $\mathfrak{c}_{0}=5$ and $\mathfrak{c}_{1}=3$, as $\mathfrak{N}_{S,3}^{1} = \{2,5\}$ and the closure of $Z_{\tau(P_{[3]})\setminus\mathfrak{N}_{S,3}^{1},P_{[3]}\setminus\mathfrak{N}_{S,3}^{1}} = Z_{34,13}$ is $Z_{134,134}$.
\end{Exa}
We will also need the following:
\begin{Lem}
	Given a non-self adjoint $Z_{S,k}$ for some fixed shape, the closure $Z_{S,k}^{C}$ will be some $Z_{S^{\prime},k^{\prime}}$, $k^{\prime}>k$, of either the shape $S$ or its restriction to a smaller $\mcO_q(H(N^{\prime}))$. Let $t_{i}$ be the maximal element of $\mathfrak{R}_{S,k}$. Then we have
	\begin{align}\label{equation: weight relation 2}
		Z_{S,k}^{C}Z_{P_{[k]}\setminus t_{i},\tau(P_{[k]})\setminus t_{i}} &=-(-q)^{k-\mathfrak{c}_{0}-1}Z_{S^{\prime},k^{\prime}-1}Z_{S,k}^{*}
	\end{align} 
\end{Lem}
\begin{proof}
	Using equation \eqref{EqMuirBr}, take $I=J$ such that $Z_{S,k}^{C}=Z_{II}$. Then take $F$ such that $I_{F}=P_{[k]}\setminus t_{i}$, and $G$ such that $J_{G}=\tau(P_{[k]})\setminus t_{i}$. Let $j$ be the smaller element of $t_{i}$. Further choose $K=\{2,...,\mathfrak{c}_{0}-k+1\}$. Then we have
	\[
	Z_{II}Z_{P_{[k]}\setminus t_{i},\tau(P_{[k]})\setminus t_{i}}=\]
	\[
	\sum\limits_{P}(-q)^{\text{wt}(P)-\text{wt}(K)}\left(\hat{R}^{-1}\right)^{I\setminus j,I\setminus j}_{P_{[k]},P_{[k]}}\hat{R}^{\tau(P_{[k]})\setminus\tau(j)\cup\{\tau(j),\tau(k+1,...,\tau(j))\}_{P},C}_{P_{[k]},P_{[k]}}Z_{I\setminus j,C}Z_{P_{[k]},\tau(P_{[k]})\setminus\tau(j)\cup\{\tau(j),\tau(k+1,...,\tau(j))\}^{P}}.\]
	The only non-zero choice of $P$ is $P=\{1,...,\mathfrak{c}_{0}-k\}$, which gives the result.
\end{proof}	
We can now relate the highest weights of arbitrary representations to those of big cell representations:
\begin{Theorem}\label{theorem: highest weights of Zsk}
	Given an irreducible representation with fixed shape of rank $M$, the highest weights of $Z_{S,k}$ are given by
	\begin{align}
		\lvert Z_{S,k}\rvert v_{0} &= \left(q^{W_{r}\left(P_{[k]}\cup\tau(P_{[k]})\right)-\lvert\mathfrak{C}_{S,k}\rvert k+\sum\limits_{i\in\mathfrak{C}_{S,k}}(\mathfrak{c}_{i}+i)}\right)v_{0}
	\end{align}
	where $r\in\mathbb{R}^{M}$ is the highest weight of a rank $M$ big cell representation with $\epsilon$ chosen such that $\epsilon_{i}=-1$ for $i\in W_{\epsilon}$.
\end{Theorem}
\begin{proof}
	Given a fixed shape $S$ of $\mcO_{q}(H(N))$, we will need to consider all $Z_{\downarrow^{m}S,k}$ appearing in the restrictions of $S$ for $0\leq m\leq N-2$. We first consider the set of self adjoint $Z_{\downarrow^{m}S,k}$. Then the closure of any non-self adjoint $Z_{S,k^{\prime}}$ will be one of the self adjoint terms. It is straightforward to see that starting with the smallest rank term and working in terms of increasing ranks, we can use the $\alpha_{i}$ coactions to permute the self adjoint terms and obtain the leading minors of a big cell representation. It follows that the highest weight of a rank $K$ self adjoint $Z_{\downarrow^{m}S,k}$ is
	\[
	\prod\limits_{i=1}^{K}\epsilon_{(0,i]}q^{2r_{i}}.\]
	Which we can rewrite as
	\[
	\lvert Z_{S,k}\rvert v_{0} = q^{2W_{r}(\downarrow^{m}P_{[k]})}v_{0}.\]
	Now consider a non-self adjoint $Z_{S,k}$. If $P_{[k]}$ contains a single $i$ such that $\tau(i)\notin P_{[k]}$, then using equation \eqref{equation: weight relation 1}, we can write:
	\[
	\epsilon_{j}q^{2W_{r}(P_{[k]}\cup\tau(i))+2W_{r}(P_{[k]}\setminus\{i\})} = -q^{-2}Z_{S,k}Z_{S,k}^{*},\]
	where $j\in W_{\epsilon}$ corresponds to $\{i,\tau(i)\}$.
	
	If $P_{[k]}$ contains more than one $i$ such that $\tau(i)\notin P_{[k]}$, then we can use equation \eqref{equation: weight relation 2}. This allows us to relate the weight of $Z_{S,k}$ to that of $Z_{S^{\prime},k^{\prime}}$ with one of the non-fixed $i$'s removed. The weight formula then follows from this.
\end{proof}
Note in the above theorem, if $\tau(i)=i$ and $u(i)=-1$, then we need to choose $\epsilon$ to account for this as well, however if the shape does not put a restriction on a particular $\epsilon_{j}$, then we can take it to be either $1$ or $-1$.

Before proceeding to the classification, we want to describe which elements may act as zero on the highest weight vector:

\begin{Def}
	For a given shape and highest weight vector of arbitrary highest weight, we denote by $\mathcal{A}_{S}$, the set
	\[
	\{Z_{ij}:Z_{ij}v_{0}=0\}.\]
\end{Def}

\begin{Prop}
	If $i\vartriangleright j$ under the $\tau$ ordering, or $u(i)=0$ and $u(j)\neq0$, then $Z_{ij}\in\mathcal{A}_{S}$.
\end{Prop}
\begin{proof}
Recall the commutation relations given by Lemma \ref{equation: restricted ZSk comm rels}. By including $Z_{S,k}$ from restrictions, we see that in either case we can find some $Z_{\downarrow^{m}S,k}$ such that its indices contain $j$, but do not contain $i$ or $\tau(i)$. It then follows from the commutation relation that $Z_{ij}$ must act as zero on $v_{0}$.
\end{proof}

We note that the above proposition does not give a full description of $\mathcal{A}_{s}$, which may contain extra elements depending on the shape. For example, for $\tau=(1,4,3,2)$, $u(i)\neq 0$, we have $Z_{23}\in\mathcal{A}_{S}$ by the $\tau$ ordering, however we also have $Z_{32}\in\mathcal{A}_{S}$ as $Z_{13,12}=0$.

\section{Shapes of finite dimensional representations}\label{section: fin dim reps}

The final step we need to take before proceeding to a general classification is to better understand the case of finite dimensional representations. In particular, we want to show the following:

\begin{Theorem}\label{theorem: fd shapes}
	Let $V$ be a finite dimensional irreducible $\mcO_q(H(N))$-representation. Then its shape $S$ must be the same as an $\mcO_q(H(N))$-character.
\end{Theorem}
We will split the proof of this into two parts, depending on whether or not $\Ker(Z_{11}) = \Ker(Z_{[1]}) = 0$ in the representation.
\begin{Prop}\label{z11 prop}
	If $V$ is an irreducible finite dimensional $\mcO_q(H(N))$ representation with $\Ker(Z_{[1]})= 0$, then $\Ker(Z_{[k]})= 0$ for all $1\leq k\leq N$.
\end{Prop}	
\begin{proof}	
	Consider first the case $N= 2$. Then by the classification in Section 2.2 of \cite{DCMo24a}, any irreducible finite dimensional $\mcO_q(H(2))$ representation with $Z_{[1]}\neq 0$ has $Z_{[2]}\neq 0$. 
	
	We now proceed by induction. Consider some irreducible finite dimensional $\mcO_q(H(N))$-representation $V$ with $Z_{[1]}\neq0$, so $\Ker(Z_{[1]})= 0$. Let $V =\oplus_i V_{i}$ be a decomposition of $V$ into irreducible $\mcO_q(H(N-1))$ representations. As $Z_{[k]}$ $q$-commutes with all $Z_{ij}\in\mcO_q(H(N))$, it follows that $Z_{[1]}\neq 0$ when acting on each $V_i$, and so we have by induction that $Z_{[k]}\neq 0$ on each $V_i$ for $1\leq k\leq N-1$, and hence $\Ker(Z_{[k]})\neq 0$ on $V$ for $1\leq k\leq N-1$. 
	
	Now as the $Z_{[j]}$ pairwise commute, they have a joint eigenbasis with non-zero eigenvalues. The $q$-commutativity with the generators $Z_{ij}$ then forces each $Z_{ij}$, for $i\neq j$, to act as a raising/lowering operator on $V$. Since $V$ is finite dimensional, we can find a lowest weight vector $v$ such that $Z_{ij}v=0$ for all $1\leq i<j\leq N$. Considering $Z_{[k]}v$, by equation \eqref{TheoCentrPol} we can then simplify to get
	\[
	Z_{[k]}v = \prod\limits_{i=1}^{k}Z_{ii}v. \]
	Hence as $\Ker(Z_{[k]})= 0$ for $1\leq k\leq N-1$, we see that  $Z_{[N]}=0$ only if $Z_{NN}v=0$. Suppose now that $Z_{NN}v=0$. Considering from equation \eqref{EqCommOqHN} the relation 
	\[
	Z_{N-1,N}Z_{N,N-1}-Z_{N,N-1}Z_{N-1,N}=(1-q^2)(Z_{NN}Z_{N-1,N-1}+\sum\limits_{l<N-1}Z_{Nl}Z_{lN}-\sum\limits_{l<N}Z_{N-1,l}Z_{l,N-1})
	\]
	acting on $v$, we get 
	\[
	Z_{N-1,N}Z_{N,N-1}v=(q^{2}-1)Z_{N-1,N-1}^{2}v. 
	\]
	Comparing both sides, we see that by the assumption that $\Ker(Z_{[N-1]}) = 0$, we have $Z_{N-1,N-1}v\neq 0$ and so the right hand side has negative coefficient. However the left hand side gives $\lvert\lambda\rvert^{2}v$, for some $\lambda\in\mathbb{C}$. This contradiction then shows that $Z_{NN}v\neq 0$, and so we must have $Z_{[N]}\neq 0$.
\end{proof} 

By the results in Section 4.4 of \cite{DCMo24a}, we then obtain:
\begin{Cor}
	If $V$ is an irreducible finite-dimensional $\mcO_q(H(N))$-representation and $\Ker(Z_{[1]})= 0$, then $V$ is a big cell representation of signature $+ = (+,+,\ldots,+)$ or $-=(-,-,\ldots,-)$. In particular, any finite dimensional $\mcO_q(H(N))$ irreducible factors through $\mcO_q(T(N))$.
\end{Cor}
In other words, any irreducible $\mcO_q(H(N))$-representation $V$ with $\Ker(Z_{[1]})=0$ is obtained by tensoring a constant character $\chi(Z) = \pm I_N$ with an admissible irreducible $U_q(\mfu(N))$-representation, and has shape $S = \pm I_N$. 

We now return to the second half of the proof of Theorem \ref{theorem: fd shapes}, for the case $Z_{11} = 0$:
\begin{proof}
	If $Z_{1j}= Z_{j1} = 0$ for all $j$, then by Lemma \ref{Lem OqHn quotient} the representation factors through $\mcO_q(H(N-1))$, and the result then follows by induction. 
	
	We may hence assume that $Z_{11}=0$ but $Z_{1j}\neq 0\neq Z_{j1}$ for some $1<j\leq N$. We can then also assume that $Z_{k1}=0$ for  all $1\leq k<j$, so $Z_{S,1}= Z_{j1}$. Now, consider from \eqref{EqCommOqHN} the relation
	\[
	Z_{1,j+1}Z_{j+1,1}-Z_{j+1,1}Z_{1,j+1}=(1-q^{2})(Z_{j+1,j+1}Z_{11}-\sum\limits_{l<j+1}Z_{1l}Z_{l1}).\]
	We see with $Z_{k1}=0$ for $1\leq k<j$ this reduces to 
	\[
	Z_{1,j+1}Z_{j+1,1}-Z_{j+1,1}Z_{1,j+1}=(q^{2}-1)Z_{1j}Z_{j1}\]
	
	Hence $Z_{j+1,1}$ acts as a raising operator. We can therefore take a basis of $V=\{v_{i,k}\}$ such that
	\[
	Z_{1,j+1}v_{i,k}=\lambda^{(1,j+1)}_{i,k}v_{i+1,k}, \quad \lambda^{(1,j+1)}_{i,k}\in\mathbb{C}, \quad \text{ and } \quad Z_{1j}v_{i,k}=q^{i}\lambda^{(1j)}_{0,k}v_{i,k}, \quad \lambda_{0,k}^{(1j)}\in\mathbb{C}^{\times}. \] 
	Using the above commutation relation, we find:
	\[
	\lvert\lambda^{(1,j+1)}_{i,k}\rvert^{2}=\lvert\lambda^{(1j)}_{0,k}\rvert^{2}(1-q^{2i+2})
	\]
	Clearly $\lambda^{(i,j+1)}_{i,k}$ is non-zero for all $i$. Then for $V$ to be finite dimensional we must have $Z_{S,1}=Z_{N1}$.
	
	Now, assume that $Z_{k2}=0$ for $2\leq k<N$ and $Z_{N2}\neq 0$. From \eqref{EqCommOqHN}, after quotienting we get the relation
	\[
	Z_{2N}Z_{N2}-Z_{N2}Z_{2N} = (1-q^2)Z_{N1}Z_{1N}.\]
	Similarly to the $Z_{j1}$ case we see then that $Z_{k2}=0$ for $2\leq k<N$ and $Z_{N2}\neq 0$ is not possible. Hence if $V$ is finite dimensional and $Z_{S,1}=Z_{N1}$, then we must have some $Z_{k2}\neq 0$ for $2\leq k\leq N-1$.
	
	Now consider restricting $V$ to $\mcO_q(H(N-1))$. If $\downarrow Z_{22}\neq 0$, then by Proposition \ref{z11 prop}, we see as an $\mcO_q(H(N-1))$ representation, $V$ must have shape
	\[
	\tau = \id, \qquad u = (0,\pm 1,...,\pm 1).\]
	By the restriction of shapes described in Proposition \ref{proposition: restriction of shapes}, we then see that as an $\mcO_q(H(N))$ representation, since $Z_{S,1}=Z_{N1}$, $V$ must have shape
	\[
	\tau = (N,2,3,...,N-1,1), \qquad u = (e^{i\theta},\pm 1,...,\pm 1,e^{-i\theta}),\]
	which is a character shape.
	
	Again considering the restriction of $V$ to $\mcO_q(H(N-1))$, if instead $\downarrow Z_{22}=0$, then by the previous arguments we must have $\downarrow Z_{N-1,2}\neq 0$. Hence $Z_{\downarrow S,1}=Z_{N-1,2}$. Again by Proposition \ref{proposition: restriction of shapes}, we see that the shape of $V$ as an $\mcO_q(H(N))$ representation must have
	\[
	\tau(1)=N, \qquad \tau(2)=N-1.\]
	Repeating this argument, we see it will fully determine the shape of $V$, and the only possible shapes that can appear are character shapes. 
\end{proof}

Let us return now to the setting of general $\mcO_q(H(N))$ representations. Consider tensoring an $\mcO_q(H(N))$ representation with a $U_q(\mfu(N))$ representation, and consider the possibilities for the shape of the resulting representation. We show in particular that the coaction by $U_q(\mfu(N))$ preserves the shape of the original representation:

\begin{Prop}
If $\pi$ is an $\mcO_q(H(N))$ representation of shape $S$, and $V$ a $U_q(\mfu(N))$ representation, then $\pi$ and $\pi\otimes V$ have the same shape. Further, under the coaction we have
\begin{equation}
	Z_{S,k}\mapsto Z_{S,k}\otimes K^{-1}_{\tau(P_{[k]})}K^{-1}_{P_{[k]}},
\end{equation}
where $K_{I}$ denotes $\prod\limits_{i\in I}K_{i}$.
\end{Prop}
\begin{proof}
To prove the shape is preserved by the $U_q(\mfu(N))$ coaction, it suffices to show the above coaction relation for $Z_{S,k}$ combined with showing no smaller $Z_{IJ}$ becomes non-zero.

Recall the coaction $\Ad^{T}_q$ given in terms of quantum minors by equation \eqref{equation: T coaction q minors}. As $T_{ij}=0$ for $i>j$, this reduces to
\[
\Ad^{T}_q: Z_{IJ}\mapsto\sum\limits_{K\leq I, L\leq J}Z_{KL}\otimes T_{KI}^{*}T_{LJ}.\]
Then the sum is over $Z_{KL}$ such that $(L,K)\leq (J,I)$ lexicographically, and so for a given rank, the minimal non-zero $Z_{IJ}$ under the action of $\Ad^{T}_q$ is $Z_{S,k} = Z_{\tau(P_{[k]}),P_{[k]}}$, and the coaction becomes
\[
\Ad^{T}_q: Z_{S,k}\mapsto Z_{S,k}\otimes T_{\tau(P_{[k]}),\tau(P_{[k]})}^{*}T_{P_{[k]},P_{[k]}} = Z_{S,k}\otimes \left(\prod\limits_{j\in \tau(P_{[k]})}T_{jj}\prod\limits_{i\in P_{[k]}}T_{ii}\right).\]
Hence the coaction $\Ad^{T}_q$ must preserve the shape of any $\mcO_q(H(N))$-representation.
\end{proof}

\section{Classification of Irreducibles}\label{section: classification}

We can now proceed to using the highest weight theory to classify the irreducibles of $\mcO_{q}(H(N))$. Recall we denote by $s\in\mathbb{R}^{N}$ the central $*$-character, i.e. the value of the centre of $\mcO_{q}(H(N))$ on a specific irreducible. We start with the following:

\begin{Theorem}\label{theorem: shape and highest weight determine rep}
Given two $\mcO_{q}(H(N))$ irreducibles $V$, $V^{\prime}$, consider their shape and central $*$-characters $(S,s)$, $(S^{\prime},s^{\prime})$ respectively. Then $(S,s)=(S^{\prime},s^{\prime})$ if and only if $V\simeq V^{\prime}$.
\end{Theorem}
\begin{proof}
It is straightforward to see that if $(S,s)\neq (S^{\prime},s^{\prime})$, then $V\not\simeq V^{\prime}$. We then only need to show $(S,s)=(S^{\prime},s^{\prime})$ implies $V\simeq V^{\prime}$.

Denote the highest weight vectors of $V$ and $V^{\prime}$ by $v_{0}$ and $v_{0}^{\prime}$ respectively. Recall from Proposition \ref{PropHC} the central $*$-characters can be obtained as the elementary symmetric polynomials over the highest weights of a big cell representation. We can reverse this to obtain a set of big cell weights from a given central $*$-character. By using the big cell weights with Theorem \ref{theorem: highest weights of Zsk}, we can state how $Z_{S,k}$ will act on $v_{0}$ and $v_{0}^{\prime}$. Further, by combining with Proposition \ref{proposition: restriction of shapes}, we can state how each $Z_{\downarrow^{m}S,k}$, as well as the $\mcO_{q}(H(N-m))$ centres, will act on $v_{0}$ and $v_{0}^{\prime}$. In particular, from the $\mcO_{q}(H(N-m))$ quantum traces, we get how $Z_{ii}$ will act on $v_{0}$ and $v_{0}^{\prime}$ for each $1\leq i\leq N$. It follows that if $(S,s)=(S^{\prime},s^{\prime})$, then the eigenvalues of any $Z_{S,k},Z_{\downarrow^{m}S,k}$, or $Z_{ii}$ will be the same whether acting on $v_{0}$ or $v_{0}^{\prime}$.

We now assume $(S,s)=(S^{\prime},s^{\prime})$, but $V\not\simeq V^{\prime}$. Then there must be some $z\in\mcO_{q}(H(N))$ such that 
\[
z^{*}zv_{0} \neq z^{*}zv_{0}^{\prime}.\]
However, as $z^{*}z$ is self-adjoint, it will simplify to give some polynomial in $Z_{S,k},Z_{\downarrow^{m}S,k}$, and $Z_{ii}$ acting on $v_{0}$ and $v_{0}^{\prime}$, which by the above must have equal values. Therefore $(S,s)=(S^{\prime},s^{\prime})$ forces us to have $V\simeq V^{\prime}$.
\end{proof}

Hence, the combination of shape and central $*$-character uniquely determines an $\mcO_{q}(H(N))$ irreducible. The possible central values for a chosen shape are given by combining Theorem \ref{theorem: highest weights of Zsk} with Theorem \ref{theorem: highest weights of big cell reps}. In particular, recalling the definitions \ref{def: central char signature} and \ref{def: shape signature} of signature for central $*$-characters and shapes respectively, we see that their signatures must be equal. All that remains now is to show that any possible irreducible actually exists.

We first note what the highest weights of characters are:

\begin{Lem}\label{lemma: character weights}
Given a non-diagonal $\mcO_{q}(H(N))$ character of rank $N$, with $\tau\neq\id$ and $k$ minimal such that $Z_{kk}\neq 0$, in terms of highest weights, the characters are given by:
\[
Z_{kk} = \epsilon_{1}q^{2r_{1}}, ~ \lvert Z_{N1}\rvert = q^{r_{1}+r_{N-2k+4}+N-2k+3}, ~ Z_{NN} = \epsilon_{1}(q^{2r_{1}}-q^{2r_{N-2k+4}+2N-4k+6}).\]
\end{Lem}
\begin{proof}
This can be calculated using Theorem \ref{theorem: highest weights of Zsk}, along with Proposition \ref{PropHC} for the value of $Z_{NN}$.
\end{proof}

We will also need the following:

\begin{Prop}\label{prop: construction of shapes}
Every shape can be constructed via the $\mcO_{q}(U(N))$ coaction on character shapes.
\end{Prop}
\begin{proof}
We first recall the statement of Theorem \ref{theorem: Uq coaction}, that the coaction changes the shape matrix $S$ by $c_{k}Sc_{k}$ (or $c_{k}S$ if $(\tau(k)=k+1)$), where $c_{k}$ is the $k$th permutation matrix. It is then straightforward to note that when we act by $\alpha_{k}$, there will be at most a $4\times 4$ sub-matrix of $S$ that is affected, corresponding to $k,\tau(k),k+1,\tau(k+1)$.

Next, we want to show we can construct all diagonal shapes. Consider a diagonal shape of signature $(N_{+},N_{-},N_{0})$. If $N_{+}N_{-}=0$, then we just take a diagonal character of rank $N-N_{0}$. If $N_{+}N_{-}\neq 0$, we instead take a non-diagonal character shape $S_{\chi}$ of rank $N-N_{0}$, with $2\min(N_{+},N_{-})$ non-diagonal entries. Let the first non-diagonal term be at the position $(i,N-i+1)$. Then consider the sub-matrix $\left(S_{\chi}\right)_{(i\leq j,k\leq N-i+1)}$. It will be of the form
\[
\left(S_{\chi}\right)_{(i\leq j,k\leq N-i+1)} = \pm \begin{pmatrix} 0 & & &  e^{-i\theta} \\ & 1 & & \\  & & \ddots & \\  e^{i\theta} & & & 0
\end{pmatrix}.	\]
 Acting $\alpha_{i}$ will shift the $e^{-i\theta}$ downwards, whilst acting $\alpha_{N-i}$ will shift it leftwards. Hence for any $i\leq j\leq N-i$, we can apply repeated coactions until the $e^{-i\theta}$ lies at the position $(j,j+1)$. The sub-matrix of the corresponding new shape matrix will now be of the form
 \[
\alpha\left(S_{\chi}\right)_{(i\leq j,k\leq N-i+1)} = \pm \begin{pmatrix} 1 & & & & & \\ &\ddots & & & & \\ & &0 &e^{-i\theta} & &\\  & &e^{i\theta} & 0 & & \\   & & & & 1 & \\ & & & & & \ddots
 \end{pmatrix}.	\] 
 Then, applying $\alpha_{j}$, we get two new shape matrices which each have an extra $1$ and $-1$ appearing on the diagonal. The extra $-1$ that appears will lie at the $(j,j)$ position in one shape matrix, and the $(j+1,j+1)$ position in the other shape matrix. We can then repeat this process until the remaining non-diagonal terms are removed, and the constructed shape matrix has the correct number of positive and negative diagonals entries. Further use of the coaction will move the zero entries to the correct place, and we will have constructed the chosen diagonal shape.
 
 For non-diagonal shapes. We first repeat the above process to obtain the correct diagonal entries, then further applications of the coaction will move any remaining non-diagonal entry to the appropriate place. Hence we can construct any shape using the coaction.
\end{proof}

We can now proceed to showing that all possible highest weight representations exist:

\begin{Theorem}\label{theorem: construction of irreducibles}
For every possible allowed pairing $(S,s)$ of self-adjoint shape and central $*$-character such that the signature of $s$ matches $S$, there is a corresponding irreducible representation of $\mcO_{q}(H(N))$.
\end{Theorem}
\begin{proof}
We begin by fixing a choice $(S,s)$ of shape and central $*$-character whose irreducible representation we want to construct. We first want to show how to construct a finite dimensional irreducible whose signature is the same as $S$, and with central $*$-character $s$. From Theorem \ref{theorem: fd shapes}, we know every finite dimensional representation has the same shape as a character. Let $\chi_{S}$ be an $\mcO_q(H(N))$ character with the same signature as $S$. Then we can use the $\mcO_q(U(N))$ coaction on $\chi_{S}$ to obtain a big cell representation. We want the coaction to be such that we obtain a big cell representation corresponding to a particular choice of $\epsilon$. If $\chi_{S}$ is a diagonal character of rank $M$, the coaction will always result in an $\epsilon$ with
\[
\epsilon_{1}=\pm 1, ~ \epsilon_{2},...,\epsilon_{M}=1, ~ \epsilon_{i}=0\text{ for }i>M.\]
If $\chi_{S}$ is a non-diagonal character of rank $M$ with $Z_{kk}$ the minimal non-zero diagonal term, we want to obtain $\epsilon$ with
\[
\epsilon_{1}=\pm 1, ~ \epsilon_{M-k+2}=-1, ~ \epsilon_{i}=1\text{ for }2\leq i\leq M, i\neq M-k+2, ~ \epsilon_{j}=0\text{ for }j>M.\]
From Proposition \ref{prop: construction of shapes}, we know it is always possible to choose such a coaction. Then our chosen big cell representation will be a subrepresentation of 
\[
\chi_{S}\otimes \mbs_{i_1}\otimes\ldots \otimes\mbs_{i_j},\]
where $\mbs_{i}$ is the $\mcO_q(U(N))$ irreducible corresponding to $\alpha_{i}$. Further, by Lemma \ref{lemma: character weights}, as well as coinvariance of the centre, we know our chosen big cell representation will have highest weight 
\[
r = (\epsilon_{1}q^{r_{1}},...,\epsilon_{1}q^{r_{1}},0,...,0), \qquad \text{ or } \qquad r = (\epsilon_{1} q^{r_{1}},...,\epsilon_{1} q^{r_{1}},-\epsilon_{1} q^{r_{M-k+2}},...,-\epsilon_{1} q^{r_{M-k+2}},0,...,0).\]
Note that what we are calling a weight here differs slightly from the original defined in Definition \ref{def: TN highest weight}, as we also want to show the corresponding $\epsilon_{(0,i]}$ values at the same time. Taking the set of exponents of $q$ will give the original $\mcO_q^{\epsilon}(T(N))$ weight that the big cell representation factors through. Alternatively, we can get the components of the elementary symmetric polynomials as in Proposition \ref{PropHC} by doubling each exponent and adding $2i-2$ to the $i$th exponent.

Let $V_{\lambda}$ be a $U_q(\mfu(N))$ irreducible of highest weight $\lambda$. Then 
\[
\chi_{S}\otimes \mbs_{i_1}\otimes\ldots \otimes\mbs_{i_j}\otimes V_{\lambda}\]
will contain a big cell subrepresentation of highest weight
\[
r+\lambda = (\epsilon_{1}q^{r_{1}+\lambda_{1}},...,\epsilon_{1}q^{r_{1}+\lambda_{M}},0,...,0),\]
or
\[
r+\lambda = (\epsilon_{1} q^{r_{1}+\lambda_{1}},\epsilon_{1}q^{r_{1}+\lambda_{2}}...,\epsilon_{1} q^{r_{1}+\lambda_{M-k+1}},-\epsilon_{1} q^{r_{M-k+2}+\lambda_{M-k+2}},...,-\epsilon_{1} q^{r_{M-k+2}+\lambda_{M}},0,...,0).\]
By comparing with the possible big cell highest weights in Theorem \ref{theorem: highest weights of big cell reps}, it follows that we can always find an $r$ and $\lambda$ that give the chosen central $*$-character $s$.

Now, let $l^{2}(\mathbb{N})=\{e_{n}\}$ be the basis of the representation $\mbs_{i}$. We can consider the limit
\[
\lim\limits_{n\rightarrow\infty}\mbs(X_{jj})e_{n},\]
which will result in the trivial $\mcO_q(U(N))$ representation. Consider applying this limiting process repeatedly to
\[
\chi_{S}\otimes \mbs_{i_1}\otimes\ldots \otimes\mbs_{i_j}\otimes V_{\lambda},\]
which will reduce to
\[
\chi_{S}\otimes V_{\lambda}.\]
That it reduces to this can be seen in terms of the matrix definition of the coactions; the limit can be viewed as replacing each copy of $X,X^{*}$ in the following with an identity matrix:
\[
T^{*}_{1,j+2}X^{*}_{1,j+1}...X^{*}_{12}\chi_{S}X_{12}...X_{1,j+1}T_{1,j+2}.\]
We know from Proposition \ref{PropHC} that the central $*$-character will act as a finite scalar on any irreducible, so it must be preserved by this limiting process. Hence it follows that $\chi_{S}\otimes V_{\lambda}$ must contain an irreducible $\mcO_q(H(N))$ representation of central $*$-character $s$. By applying Proposition \ref{prop: construction of shapes} to this irreducible, as well as noting again that the centre is coinvariant under $\alpha_{k}$, it follows that we will have constructed an irreducible representation of shape $S$ and central $*$-character $s$.
\end{proof}

From Corollary 3.7 of \cite{DCMo24a}, we know an irreducible $\mcO_q(H(N))$ representation will be bounded as soon as its central $*$-character is bounded. We can therefore summarize the combination of results of Theorems \ref{theorem: shape and highest weight determine rep} and \ref{theorem: construction of irreducibles}:
\begin{Theorem}
Every irreducible bounded $*$-representation of $\mcO_{q}(H(N))$ is uniquely classified by the pair $(S,s)$, where $S$ is a self-adjoint shape, and $s$ is a central $*$-character of a big cell representation of the same signature as $S$.
\end{Theorem}

The following was shown as part of the proof of Theorem \ref{theorem: construction of irreducibles}, but is worth mentioning on its own:
\begin{Cor}
Every finite dimensional $\mcO_q(H(N))$ irreducible appears as a subrepresentation of $\chi\otimes V$, for some $\mcO_q(H(N))$ character $\chi$, and $U_q(\mfu(N))$ irreducible $V$.
\end{Cor}
\begin{Rem}
We can actually go one step further than the above corollary, and explain explicitly how to choose a character and $U_q(\mfu(N))$ irreducible to produce a particular finite dimensional $\mcO_q(H(N))$ representation. We first note the following: Let $r$, $r^{\prime}$, be weights where $r$ is $\epsilon$-adapted, and $r^{\prime}$ is $\epsilon^{\prime}$-adapted, with $\epsilon\neq\epsilon^{\prime}$ of the same signature. Then if $r$ and $r^{\prime}$ produce the same central $*$-character, we must be able to relate $r$ and $r^{\prime}$ via "shifted" permutations, i.e. maps of the form
\[
\epsilon_{(0,i]}q^{r_{i}}\mapsto \epsilon^{\prime}_{(0,j]}q^{r^{\prime}_{j}+(j-i)}, \qquad \epsilon_{(0,j]}q^{r_{j}}\mapsto \epsilon^{\prime}_{(0,i]}q^{r^{\prime}_{i}-(j-i)}, \qquad i<j.\]
Hence, given the highest weight of a finite dimensional $\mcO_q(H(N))$ irreducible of the form in Theorem \ref{theorem: highest weights of Zsk}, we can apply shifted permutations to it to obtain a highest weight such that the positive signature weights (respectively negative signature weights) are consecutive terms. The choice of character weight and $U_q(\mfu(N))$ weight can then be taken directly from this.
\end{Rem}
\begin{Exa}
The finite dimensional $\mcO_q(H(4))$ shapes of rank $4$ are as follows:
\[
\pm\begin{pmatrix} 1& & & \\ &  1& & \\ & &  1& \\ & & & 1 \end{pmatrix}, \qquad \pm\begin{pmatrix} 0& & & e^{-i\theta} \\ &  1& & \\ & &  1& \\ e^{i\theta} & & & 0 \end{pmatrix}, \qquad \begin{pmatrix} 0& & &e^{-i\theta_{1}} \\ &0 & e^{-i\theta_{2}}& \\ & e^{i\theta_{2}} & 0 & \\ e^{-\theta_{1}} & & & 0 \end{pmatrix}.\]
The first case is just the $\mcO_q(T(4))$ embedding, so we will show how to construct the second and third cases. By Theorem \ref{theorem: highest weights of Zsk}, the second shape will correspond to $\epsilon=\{\pm 1,1,1,-1\}$, and so will have highest weights of the form \[
(\epsilon_{1}q^{r_{1}},\epsilon_{1}q^{r_{2}},\epsilon_{1}q^{r_{3}},-\epsilon_{1}q^{r_{4}}), \qquad r_{1}\leq r_{2}\leq r_{3}.\]
In this case, we already have $\epsilon$ in the correct form, so we see that we simply need to choose a character $\chi$ and $\mcO_q(T(4))$ irreducible $V_{\lambda}$ with weights
\[
(\epsilon_{1}q^{t_{1}},\epsilon_{1}q^{t_{1}},\epsilon_{1}q^{t_{1}},-\epsilon_{1}q^{t_{4}}), \qquad (q^{\lambda_{1}},q^{\lambda_{2}},q^{\lambda_{3}},q^{\lambda_{4}}),\]
respectively such that
\[
(r_{1},r_{2},r_{3},r_{4}) = (t_{1}+\lambda_{1},t_{1}+\lambda_{2},t_{1}+\lambda_{3},t_{4}+\lambda_{4}).\]
For example, we could simply take
\[
\text{h.w.}(\chi) = (\epsilon_{1}q^{0},\epsilon_{1}q^{0},\epsilon_{1}q^{0},-\epsilon_{1}q^{r_{4}}), \qquad \text{h.w.}(V_{\lambda}) = (q^{r_{1}},q^{r_{2}},q^{r_{3}},q^{0}).\]
For the third shape, again by Theorem \ref{theorem: highest weights of Zsk} we see it will correspond to $\epsilon = \{\pm 1,-1,1,-1\}$, and so have highest weights of the form
\[
(\epsilon_{1}q^{r_{1}},-\epsilon_{1}q^{r_{2}},\epsilon_{1}q^{r_{3}},-\epsilon_{1}q^{r_{4}}), \qquad r_{1}\leq r_{3}+1, \qquad r_{2}\leq r_{4}+1.\]
This time, $\epsilon$ is not in the form we want, so we will need to apply a shifted permutation to the weight, to get
\[
(\epsilon_{1}q^{r_{1}},-\epsilon_{1}q^{r_{2}},\epsilon_{1}q^{r_{3}},-\epsilon_{1}q^{r_{4}})
\mapsto (\epsilon_{1}q^{r_{1}},\epsilon_{1}q^{r_{3}+1},-\epsilon_{1}q^{r_{2}-1},-\epsilon_{1}q^{r_{4}})
 = (\epsilon_{1}q^{r_{1}},\epsilon_{1}q^{r_{1}+n_{1}+1},-\epsilon_{1}q^{r_{2}-1},-\epsilon_{1}q^{r_{2}+n_{2}}),\]
for $n_{1},n_{2}\in\mathbb{N}_{>0}$. We need to similarly shift the $\mcO_q(H(4))$ character weight, which becomes
\[
(\epsilon_{1}q^{t_{1}},-\epsilon_{1}q^{t_{2}},\epsilon_{1}q^{t_{1}+1},-\epsilon_{1}q^{t_{2}+1})\mapsto (\epsilon_{1}q^{t_{1}},\epsilon_{1}q^{t_{1}+2},-\epsilon_{1}q^{t_{2}-1},-\epsilon_{1}q^{t_{2}+1}),\]
and so we know $\chi\otimes V_{\lambda}$ will contain an $\mcO_q(H(4))$ irreducible of highest weight
\[
(\epsilon_{1}q^{t_{1}+\lambda_{1}},\epsilon_{1}q^{t_{1}+\lambda_{2}+2},-\epsilon_{1}q^{t_{2}+\lambda_{3}-1},-\epsilon_{1}q^{t_{2}+\lambda_{4}+1}).\]
Here we need to be careful with our choice of $t$ and $\lambda$, to ensure that we have $\lambda_{1}\leq \lambda_{2}\leq \lambda_{3}\leq\lambda_{4}$. We will also need to reverse the shifted permutation we applied to $\chi$. For example, we can choose 
\[
\text{h.w.}(\chi) = (\epsilon_{1}q^{r_{1}},-\epsilon_{1}q^{r_{2}-n_{1}+1},\epsilon_{1}q^{r_{1}+1},-\epsilon_{1}q^{r_{2}-n_{1}+2}), \qquad \text{h.w.}(V_{\lambda}) = (q^{0},q^{n_{1}-1},q^{n_{1}-1},q^{n_{1}+n_{2}-2}).\]
\end{Exa}

\section*{Acknowledgements}
The author thanks Kenny De Commer for introducing him to the reflection equation algebra, and for many useful comments and discussions.

\end{document}